\documentclass[10pt,a4paper]{article}

\usepackage{amsmath}
\usepackage{amsfonts}
\usepackage{amssymb}

\newtheorem{theorem}{Theorem}
\newtheorem{corollary}{Corollary}
\newtheorem{proposition}{Proposition}
\newtheorem{problem}{Problem}
\newtheorem{lemma}{Lemma}

\newtheorem{definition}{Definition}
\newtheorem{acknowledgements}{acknowledgements}

\DeclareMathOperator{\sgn}{sgn}
\newenvironment{proof}{\paragraph{{\it Proof}}}{\hfill$\square$}

\usepackage[utf8]{inputenc}

\usepackage{cite}

\usepackage[right]{lineno}

\usepackage{microtype}
\DisableLigatures[f]{encoding = *, family = * }

\raggedright
\usepackage{changepage}

\usepackage[aboveskip=1pt,labelfont=bf,labelsep=period,singlelinecheck=off]{caption}

\makeatletter
\renewcommand{\@biblabel}[1]{\quad#1.}
\makeatother

\usepackage{lastpage,fancyhdr,graphicx}
\fancyheadoffset[L]{2.25in}
\fancyfootoffset[L]{2.25in}

\usepackage{color}

\definecolor{Gray}{gray}{.25}

\usepackage{graphicx}

\usepackage{sidecap}

\usepackage{wrapfig}
\usepackage[pscoord]{eso-pic}
\usepackage[fulladjust]{marginnote}
\reversemarginpar

\begin{document}

%\title{}

% title goes here:
\begin{center}
{\Large
\textbf\newline{Topological Representation of the Transit Sets of 
       $k$-Point Crossover Operators}
}
\newline
% authors go here:

Manoj Changat\footnote{Department of Futures Studies, 
  University of Kerala, Trivandrum, IN 695 581, India, \{mchangat,ferdows.h.n,rshilpam,cyanabisha\}@gmail.com},
        Prasanth G.\ Narasimha-Shenoi\footnote{Department of Mathematics, 
  Government College, Chittur, Palakkad, IN 678 104, India, prasanthgns@gmail.com},
        Ferdoos Hossein Nezhad\textsuperscript{1},\\
        Matja{\v{z}} Kov{\v{s}}e\footnote{School of Basic Sciences, IIT Bhubaneswar, Bhubaneswar, India, kovse@iitbbs.ac.in},
        Shilpa Mohandas\textsuperscript{1}, 
        Abisha Ramachandran\textsuperscript{1},
        Peter F.\ Stadler %\bigskip
\footnote{Bioinformatics Group, Department of Computer Science \&
  Interdisciplinary Center for Bioinformatics, Universit{\"a}t Leipzig,
  H{\"a}rtelstra{\ss}e 16-18, D-04107 Leipzig, Germany; 
  German Centre for Integrative Biodiversity Research 
  (iDiv) Halle-Jena-Leipzig, Competence Center for Scalable Data
  Services and Solutions Dresden-Leipzig, Leipzig Research Center for 
  Civilization Diseases, and Centre for Biotechnology and Biomedicine at 
  Leipzig University at Universit{\"a}t Leipzig;
  Max Planck Institute for Mathematics in the Sciences, Inselstra{\ss}e 22,
  D-04103 Leipzig, Germany; 
  RNomics Group, Fraunhofer Institute for Cell Therapy and Immunology,
  D-04103 Leipzig, Germany;
  Institute for Theoretical Chemistry, University of Vienna,
  W{\"a}hringerstrasse 17, A-1090 Wien, Austria;
  Santa Fe Insitute, 1399 Hyde Park Rd., Santa Fe NM 87501, USA., studla@bioinf.uni-leipzig.de}

%\bigskip
 
\end{center}


%\maketitle

\begin{abstract}
\justify
  $k$-point crossover operators and their recombination sets are studied
  from different perspectives. We show that transit functions of $k$-point
  crossover generate, for all $k>1$, the same convexity as the interval
  function of the underlying graph. This settles in the negative an open
  problem by Mulder about whether the geodesic convexity of a connected
  graph $G$ is uniquely determined by its interval function $I$.  The
  conjecture of Gitchoff and Wagner that for each transit set $R_k(x,y)$
  distinct from a hypercube there is a unique pair of parents from which it
  is generated is settled affirmatively.  Along the way we characterize
  transit functions whose underlying graphs are Hamming graphs, and those
  with underlying partial cube graphs.  For general values of $k$ it is
  shown that the transit sets of $k$-point crossover operators are the
  subsets with maximal Vapnik-Chervonenkis dimension. Moreover, the transit
  sets of $k$-point crossover on binary strings form topes of uniform
  oriented matroid of VC-dimension $k+1$. The Topological Representation
  Theorem for oriented matroids therefore implies that $k$-point crossover
  operators can be represented by pseudosphere arrangements. This provides
  the tools necessary to study the special case $k=2$ in detail.

%%  \keywords{Genetic algorithms \and
%    recombination \and 
%    transit functions \and
%    betweenness \and
%    oriented Matroids \and
%    Vapnik-Chervonenkis dimension}
\end{abstract}

\section{Introduction}

Crossover operators are a crucial component of Genetic Algorithms and
related approaches in Evolutionary Computation. Their purpose is to combine
the genetic information of two parents to produce one or more offsprings
that are ``mixtures'' of their parents. In this contribution we will be
concerned with the specific setting of crossover operators for strings of
fixed length $n$ over an alphabet $\mathcal{A}$ of $a\ge 2$ letters. Given
two parental strings $x=(x_1x_2\dots x_n)$ and $y=(y_1y_2\dots y_n)$ one
may for instance construct recombinant offsprings of the form $(x_1x_2\dots
x_iy_{i+1}y_{i+2}\dots y_n)$ and $(y_1y_2\dots y_ix_{i+1}x_{i+2}\dots
x_n)$. The index $i$ serves as a breakpoint at which the two parents
recombine. This so-called one-point crossover can be generalized to two or
more breakpoints.
\begin{definition}\label{def:xover}
  Given $x,y,z\in\mathcal{A}^n$ we say that $z$ is a $k$-point crossover
  offspring of $x$ and $y$ if there are indices
  $0=i_0\le i_1\le i_2 \le i_k=n$ so that
  for all $\ell$, $1\le\ell\le k$, either
  $z_j=x_j$ for all $j\in\{i_{\ell-1}+1,\dots, i_{\ell} \}$ or
  $z_j=y_j$ for all $j\in\{i_{\ell-1}+1,\dots, i_{\ell} \}$.
\end{definition}
Note that this definition states that $x$ and $y$ are broken up into
\emph{at most} $k$ intervals that are alternately included into $z$.  This
convention simplifies the mathematical treatment considerably and also
conforms to the usual practice of including crossovers with fewer than the
maximum number of breakpoints. Uniform crossover, where each letter $z_i$
is freely chosen from one of the two parents, is obtained by allowing
$k=n-1$ breakpoints. We note, furthermore, that our definition ensures that
the parental strings are also included in the set of possible offsprings.

Properties of $k$-point crossover have been studied extensively in the
past. \cite{English:97} described key algebraic properties and
isomorphisms between the search spaces induced by crossover and mutation
with small populations have been analyzed by \cite{Culberson:94}. A formal
treatment of multi-point crossover with an emphasis on disruption analysis
can be found in \cite{DeJong:92}. A general review of genetic algorithms
from the perspective of stochastic processes on populations can be found in
\cite{Schmitt:01}. In this context, crossover operators are represented by
stochastic matrices. A similar matrix-based formalism is explored in
\cite{Stadler:00a}. Coordinate transformations, more precisely the Walsh
transform \cite{Goldberg:89} and its generalizations to non-binary
alphabets \cite{Field:95} have played an important role in explaining the
functioning of GAs in terms of building blocks and the Schema theorem
\cite{Holland:75}.  As a generalization, an abstract treatment of crossover
in terms of equivalence relations has been given by \cite{Radcliffe:94}.

\cite{Gitchoff:96} proposed to consider the function $R: X\times X\to 2^X$
that assigns to each possible pair of parents the set of all possible
recombinants. They asked which properties of $R$ could be used to
characterize crossover operators in general and explored properties of
$k$-point crossover on strings. In particular, they noted the following
four properties:
\begin{description}
\item[(T1)]  $x,y\in R(x,y)$ for all $x,y\in X$,
\item[(T2)]  $R(x,y)=R(y,x)$ for all $x,y\in X$,
\item[(T3)]  $R(x,x)=\{x\}$ for all $x\in X$,
\item[(GW4)] $z\in R(x,y)$ implies $|R(x,z)|\le|R(x,y)|$.
\end{description}
Mulder introduced the concept of \emph{transit functions}
characterized by the axioms (T1), (T2), and (T3) as a unifying approach to
intervals, convexities, and betweenness in graphs and posets in last decade
of the 20th century. Available as preprint only but frequently cited for
more than a decade, the seminal paper was published only recently
\cite{Mulder:08}. For example, given a connected graph $G$, its geodetic
intervals, i.e., the sets of vertices lying on shortest paths between a
pair of prescribed endpoints $x,y\in V(G)$ form a transit function usually
denoted by $I_G(x,y)$ \cite{Mulder:80} and referred as the interval function of a graph $G$. Unequal crossover, where (T3)
is violated, has been rarely explored in the context of evolutionary
computation, which the exception of the work by \cite{Shpak:99}. In this
contribution we restrict ourselves exclusively to the simpler case of
homologous string recombination.  \emph{Thus, from here on we will assume
  that $R$ satisfies (T1), (T2), and (T3).}

A common interpretation of transit functions is to view $R(x,y)$ as the
subset of $X$ lying \texttt{between} $x$ and $y$. Indeed, a transit
function is a \emph{betweenness} if it satisfies the two additional axioms
\begin{description}
\item[(B1)] $z\in R(x,y)$ and $z\ne y$ implies $y\notin R(x,z)$.
\item[(B2)] $z\in R(x,y)$ implies $R(x,z)\subseteq R(x,y)$.
\end{description}
It is natural, therefore, to regard a pair of distinct points $x$ and $y$
without other points between them as \emph{adjacent}. The corresponding
graph $G_R$ has $X$ as its vertex set and $\{x,y\}\in E(G_R)$ if and only
if $R(x,y)=\{x,y\}$ and $x\ne y$. The graph $G_R$ is known as the
\emph{underlying graph} of $R$.

\cite{Moraglio:04} introduced the notion of \emph{geometric crossover
  operators} relative to a connected reference graph $G$ with vertex set
$X$ by requiring -- in our notation -- that $R(x,y)\subseteq I_G(x,y)$ for
all $x,y\in X$. In the setting of \cite{Moraglio:04}, the reference graph
$G$ was given externally in terms of a metric on $X$. When studying
crossover in its own right it seems natural to consider the transits sets
of $R$ in relation to the intervals of $G_R$ itself.  Hence we say that $R$
is \emph{MP-geometric} if
\begin{description}
  \item[(MG)] $R(x,y)\subseteq I_{G_R}(x,y)$ for all $x,y\in X$.
\end{description}
Note the condition (MG) is an axiom for transit functions independent of
any externally prescribed structure on $X$. \cite{Mulder:08} considered a
different notion of ``geometric'' referring transit functions that satisfy
(B2) and the axiom
\begin{description}
\item[(B3)] $z\in R(x,y)$ and $w\in R(x,z)$ implies $z\in R(w,y)$.
\end{description}
Mulder's version of ``geometric'' is less pertinent for our purposes
because crossover operators usually violate (B2).

Another interpretation of $R$, which is just as useful in the context of
crossover operators, is to regard $R(x,y)$ as the set of offsprings
reachable from the parents $x$ and $y$ in a single generation. It is
natural then to associate with $R$ a function $\widehat{R}: X\times X\to
2^X$ so that $z\in \widehat{R}(x,y)$ if and if only $z$ eventually can be generated
from $x$ and $y$ and all their following generations of
offsprings. Formally, $z\in \widehat{R}(x,y)$ if there is a finite sequence
of pairs $\{x_k,y_k\}$ so that $z\in R(x_m,y_m)$, $\{x_k,y_k\}\in
R(x_{k-1},y_{k-1})$ for all $k=1,\dots, m$, $x_0=x$, and $y_0=y$.  By
construction, $R(x,y)\subseteq \widehat{R}(x,y)$ for all $x,y\in X$.  If
$R$ is a transit function, then $\widehat{R}$ is also a transit function.

We say that $R(x,y)$ is \emph{closed} if $R(x,y)=\widehat{R}(x,y)$.
Equivalently, a transit set $R(x,y)$ is closed if and only if
$R(u,v)\subseteq R(x,y)$ holds for all $u,v\in R(x,y)$, since in this case
nothing can be generated from the children of $x$ and $y$ that is not
accessible already from $x$ and $y$ itself. In particular, all singletons
and all adjacencies, i.e., individual vertices and the edges of $G_R$, are
always closed. A transit function $R$ is called \emph{monotone} if it
satisfies
\begin{description}
\item[(M)] For all $x,y\in X$ and $u,v\in R(x,y)$ implies $R(u,v)\subseteq
  R(x,y)$,
\end{description}
i.e., if all transit sets are closed. By construction, $\widehat{R}$
satisfies (M) for any transit function $R$. A simple argument\footnote{(i)
  $R(x,y)\subseteq \widehat{R}(x,y)$ by definition, (ii) $R(x,y)=\{x,y\}$
  implies $\widehat{R}(x,y)=\{x,y\}$, (iii) if $\widehat{R}(x,y)=\{x,y\}$
  but $R(x,y)\ne\{x,y\}$ either (i) or axiom (T1) is violated.} shows that
$\widehat{R}(x,y)=\{x,y\}$ if and only if $R(x,y)=\{x,y\}$. Thus $R$ and
$\widehat{R}$ have the same underlying graph $G_{\widehat{R}}=G_R$. The
sets $\{\widehat{R}(x,y) | x,y\in X \}$, finally, generate a convexity
$\mathfrak{C}_R$ consisting of all intersections of the (finitely many)
transit sets $\widehat{R}(x,y)$.

One of the most fruitful lines of research in the field of transit
functions is the search for axiomatic characterizations of a wide variety
of different types of graphs and other discrete structures in terms of
their transit functions.  \cite{Nebesky:01} showed that a function $I: V
\times V \rightarrow 2^V$ is the geodesic interval function of a
connected graph if and only if $I$ satisfies a set of axioms that are
phrased in terms of $I$ only.  Later, \cite{MulderNebesky:09} improved the
axiomatic characterization of $I(u,v)$ by formulating a nice set of
(minimal) axioms. The all-paths function $A$ of a connected graph $G$
(defined as $A(u,v)= \{z\in V(G): z$ lies on some $u,v$-path in $G\}$)
admits a similar axiomatic characterization \cite{Changat:01}.  These
results immediately raise the question whether other types of transit
functions can be characterized in terms of transit axioms only.

Since $k$-point crossover on strings over a fixed alphabet forms a rather
specialised class of recombination operators we ask here whether it can
be defined completely in terms of properties of its transit function
$R_k$. Beyond the immediate interest in $k$-point crossover operators we
can hope in this manner to identify generic properties of crossover
operators also on more general sets $X$.

This contribution is organised as follows. In section~\ref{sect:HG} we
consider transit functions whose underlying graphs $G_{R}$ are Hamming
graphs since, as we show in section~\ref{sect:basic}, $k$-point crossover
belongs to this class. We then investigate the properties of $k$-point
crossover in more detail from the point of view of transit functions.  In
section~\ref{sect:GTk} we switch to a graph-theoretical perspective and
derive a complete characterization of $k$-point crossover on binary
alphabets, making use of key properties of partial cubes. In order to
generalize these results we consider topological aspects of $k$-point
crossover in section~\ref{sect:topological} and explore its relationship
with oriented matroids. We conclude our presentation with several open
questions.

\section{Hamming Graphs and their Geodesic Intervals}
\label{sect:HG}

In most applications, $k$-point crossover will be applied to binary strings
or, less frequently, to strings over a larger, fixed-size alphabet
$\mathcal{A}$. In a population genetics context, however, the number of
allels may be different for each locus, hence we consider the most general
case here, where each sequence position is taken from a distinct alphabet
$\mathcal{A}_i$ with $a_i:=|\mathcal{A}_i|\ge2$ for $1\le i\ne n$. The
Hamming graph $\prod_i K_{a_i}$ is the Cartesian products of complete
graphs $K_{a_i}$ with $a_i$ vertices; we refer to the book by
\cite{handbookproductgraphs:2011} for more details on Hamming graphs and
product graphs in general. The special case $a_i=2$ for all $i$ is usually
called $n$-dimensional hypercube $K_2^n$. The shortest path distance on
$\prod_i K_{a_i}$ is the Hamming distance $d(x,y)$, which counts the number
of sequence positions at which the string $x$ and $y$ differ.

Given a transit function $R$ and a point $x \in X$ let $\delta(x)=|\{y \in
X \quad | \quad |R(x,y)|=2\}|$, i.e., $\delta(x)=\delta_R(x)$ is the degree
of $x$ in the underlying graph $G_R$. We write $\delta(R) = \max_{x\in X}
\delta(x)$ for the maximal degree of the underlying graph.

The purpose of this section is to characterize transit functions whose
underlying graphs are Hamming graphs. Our starting point is the following
characterization of hypercubes, which follows from results by
\cite{Mulder:80} and \cite{Laborde:82}:
\begin{proposition}
\label{prop:hypercube}
  Suppose $G$ is connected and each pair of distinct adjacent edges lies in
  exactly one 4-cycle. Then $G$ is isomorphic to $n$-dimensional hypercube
  if and only if the minimum degree $\delta$ of $G$ is finite and $|V(G)|=
  2^{\delta}$.
\end{proposition}

\begin{figure}
\begin{center}
  \includegraphics[width=0.6\textwidth]{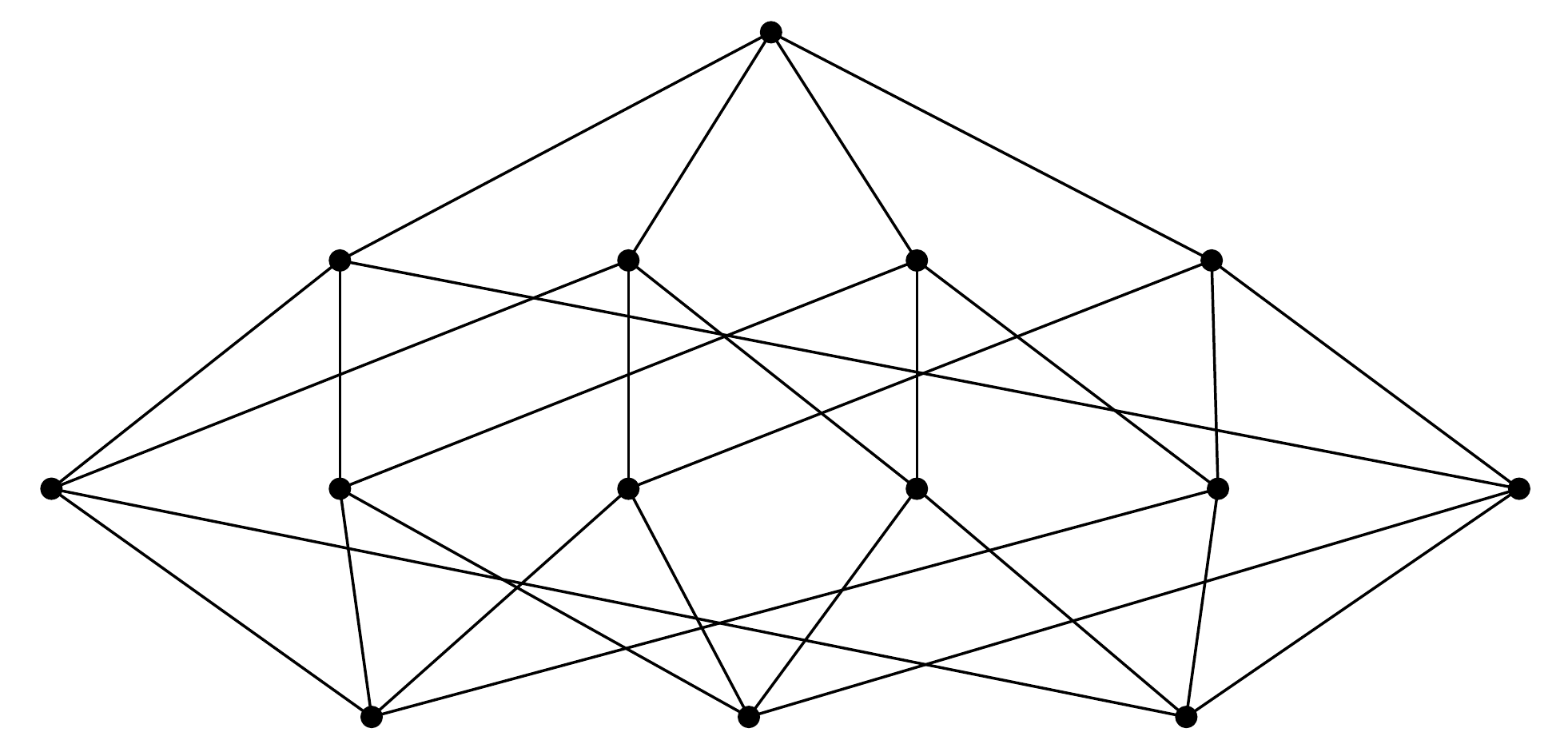}
  \caption{The last condition of Proposition \ref{prop:hypercube}, i.e.,
    $|V(G)|= 2^{\delta}$, is necessary as demonstrated by this example of a
    $(0,2)$-graph that is not a hypergraph \cite{Mulder:80}.  It
    satisfies all but the last requirement from the proposition.}
\label{fig:02graph}
\end{center}
\end{figure}

Graphs with the property that any pair of vertices has zero or exactly 2
common neighbours are called $(0,2)$-graphs \cite{Mulder:80}. We note
that the condition $|V(G)|=2^{\delta}$ in Proposition~\ref{prop:hypercube}
is necessary as demonstrated by the example in Fig.~\ref{fig:02graph}.
Proposition \ref{prop:hypercube} can be translated into the language of
transit functions as follows:
\begin{corollary}
\label{cor:hypercube}
  Let $R$ be a transit function on a set $X$ with a connected underlying
  graph. Then the underlying graph $G_R$ is isomorphic to $n$-dimensional
  hypercube $K_2^n$ if and only if $R$ satisfies:
  \begin{description}
  \item[{\rm(A1)}] For every $x,u,v$ such that $|R(x,u)|=|R(x,v)|=2$ there exist
    unique $y$ such that $|R(y,u)|=|R(y,v)|=2$,
  \item[{\rm(A2)}]  $\delta(R)=n$ and $|X|=2^{n}$.
  \end{description}
\end{corollary}
Later, \cite{mollard:91} generalized Proposition~\ref{prop:hypercube} to
arbitrary Hamming graphs. For any vertex $x$ in the graph $G$ let $N_i(x)$
denote the number of maximal $i$-cliques $K_i$ in $G$ that contain the
vertex $x$.

\begin{proposition}[\cite{mollard:91}]
\label{prop:HammingGcharaterization}
  Let $G$ be a simple connected graph such that two non-adjacent vertices
  in $G$ either have exactly $2$ common neighbors or none at all, and
  suppose $G$ has neither $K_4 \setminus e$ nor $K_2 \Box K_3 \setminus e$
  (Figure \ref{fig:forbiddensubgraphsHamming}) as induced subgraph.  Then
  $N_i(x)$ is independent of $x$ and $G$ is isomorphic to the Hamming graph
  if and only if $|V(G)|=\prod_{h=1}^{p}h^{N_i(x)}$, where $p$ is the
  maximum integer such that $N_p(x)$ is nonzero.
\end{proposition}

\begin{figure}[h]
\begin{center}
  \includegraphics[width=0.6\textwidth,clip=]{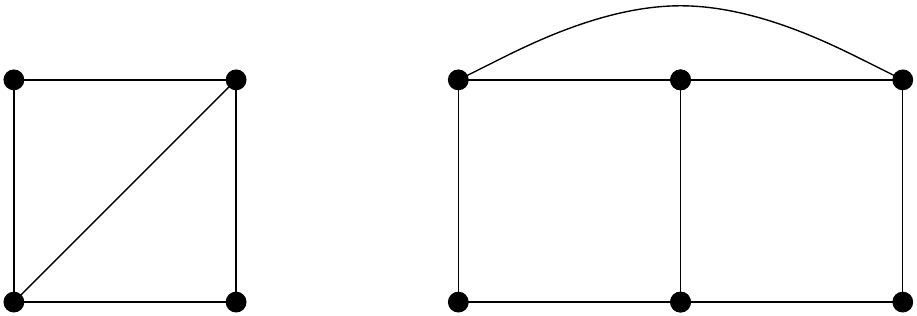}
  \caption{The forbidden induced subgraphs $K_4 \setminus e$ and $K_2 \Box
    K_3 \setminus e$ appearing in Proposition
    \ref{prop:HammingGcharaterization}.}
\label{fig:forbiddensubgraphsHamming}
\end{center}
\end{figure}

These results can again be translated into the language of transit
functions:
\begin{corollary}
  \label{cor:A1-A4}
  Let $R$ be a transit function with a connected underlying graph. Then the
  underlying graph $G_R$ is isomorphic to Hamming graph $K_a^n$ if and only
  if $R$ satisfies:
  \begin{description}
  \item[{\rm(A1)}] For every $x,u,v$ such that $|R(x,u)|=|R(x,v)|=2$ there
    exists unique $y$ such that $|R(y,u)|=|R(y,v)|=2$,
  \item[{\rm(A2')}]  $\delta(R)=n(a-1)$ and $|X|=a^{n}$,
  \item[{\rm(A3)}] There exist no $x,y,u,v$ such that
    $|R(x,u)|=|R(x,v)|=|R(y,u)|=|R(y,v)|=|R(x,y)|=2$ and $|R(u,v)|>2$,
  \item[{\rm(A4)}] There exist no $x,y,u,v, w, z$ such that\\
    $|R(x,u)|=|R(x,v)|=|R(y,u)|=|R(y,v)|=|R(v,w)|=|R(y,z)|=|R(w,z)|=|R(x,w)|=2$
    and 
    $|R(u,v)|, |R(u,w)|,|R(u,z)|,|R(x,y)|,|R(x,z)|,|R(v,z)|,|R(y,w)| > 2$.
\end{description}
\end{corollary}

The representation of Hamming graphs as $n$-fold Cartesian products of
complete graphs $H=\prod_{i=1}^n K_{a_i}$ implies a ``coordinatization'',
that is, a labeling of the vertices the reflects this product structure.
The geodesic intervals in Hamming graphs then have very simple description:
\begin{equation}
  I_{H}(x,y) = \left\{ z = (z_1,z_2,\dots z_n) \big| z_i\in \{x_i,y_i\} \textrm{ for } 
  1\le i\le n \right\}
\end{equation}
where $(x_1,x_2,\dots x_n)$ and $(y_1,y_2,\dots y_n)$ are the coordinates
of the vertices $x$ and $y$.  Thus $G_{I_H(x,y)}$ is a subhypercube of
dimension $d(x,y)$ as shown in example in the book by
\cite{handbookproductgraphs:2011}.  The intervals of Hamming graphs have
several properties that will be useful for our purposes.  A graph is called
antipodal if for every vertex $v$ there is a unique ``antipodal vertex''
$\bar v$ with maximum distance from $v$.

\begin{lemma}
  Let $Q$ be an induced sub-hypercube of a Hamming graph $H$.  Then for
  every $x\in Q$ there is a unique vertex $\bar x\in Q$ so that $Q =
  I_{H}(x,\bar x)$.
\end{lemma}
\begin{proof}
  This follows from a well known fact that hypercubes are antipodal graphs
  \cite{Mulder:80}.
  \
\end{proof}
It is well known that $I_H$ satisfies the monotone axiom (M) and thus
also (B2).

\begin{lemma}
  Let $Q'$ and $Q''$ be two induced sub-hypercubes in a Hamming graph
  $H$. Then $Q'\cap Q''$ is again an induced (possibly empty) sub-hypercube
  of $H$.
\end{lemma}
\begin{proof}
  For every coordinate $i$, $Q_i'=\{x_i|x'\in Q\}$ and $Q_i''=\{x_i|x\in
  Q''\}$ contains at most two different letters from the alphabet
  $\mathcal{A}_i$. $Q'\cap Q''=\prod_i (Q_i'\cap Q_i'')$, and hence a
  hypercube.
  \
\end{proof}
As an immediate consequence we note that $I_H$ satisfies even the stronger
property
\begin{description}
\item[(MM)] For all $u,v,x,y\in X$ holds: if $R(u,v)\cap
  R(x,y)\ne\emptyset$ then there are $p,q\in X$ so that $R(u,v)\cap
  R(x,y)=R(p,q)$.
\end{description}

The disadvantage of the results so far is that we have to require
explicitly that $G_{R}$ is connected. In the light of condition (MG) above
it seems natural to require connectedness of $G_{R}$ for recombination
operators in general.

To-date, only sufficient conditions for connectedness of $G_{R}$ are
known. Following ideas outlined by \cite{vandeVel:83}, we can show
directly that the following property is sufficient:
\begin{description}
\item[(CG)] For all $a,x,y,z \in X$:
  If $R(a,x) \subseteq R(a,y)$, then
  $R(a,x) \subseteq R(a,z)\subseteq R(a,y)$ if and only if $z \in R(x,y)$.
\end{description}
As a technical device we will employ the partial order $\leq_{a}$ of $X$
defined, for given $a\in X$, by $x \leq_{a} y$ if and only if $R(a,x)
\subseteq R(a,y)$. As usual, we write $x <_{a}y$ if $x \leq_{a} y$ and
$x\neq y$. For $R=I_G$ we have the equivalence $x \in I_G(a,y)$ if and if only
$x\leq_{a} y$.

\begin{lemma}\label{thm:connected}
  The underlying graph $G_R$ of a transit function $R$ is connected if $R$
  satisfies axiom {\rm(CG)}.
\end{lemma}
\begin{proof}
  Let $R$ be a transit function satisfying axiom (CG). Let $a,b \in X$ be
  two distinct elements, and let $\mathbf{C}=(a=a_0,a_1,\dots a_t=b)$ be a
  maximal $\leq_a$-chain between $a$ and $b$, where the elements are
  labeled in increasing order $a=a_0 <_a a_1 <_a a_2 <_a \ldots <_a a_t=b$.

  We claim that, for any $i$, $0\leq i \leq n$, elements $a_i$ and
  $a_{i+1}$ form an edge in $G_R$. To see this assume that, on the
  contrary, there is an element $x \in R(a_i,a_{i+1}) \setminus
  \{a_i,a_{i+1}\}$ for some $i$.  Then (CG) implies $R(a,a_i) \subseteq
  R(a,x)\subseteq R(a,a_{i+1})$, i.e., $a_i <_a x <_a x_{i+1}$,
  contradicting maximality of the chain $\mathbf{C}$. Hence $\mathbf{C}$
  consists of consecutive edges whence $G_R$ is a connected graph.
  \
\end{proof}

However, property (CG) is much too strong for our purposes: Setting $x=a$
makes the condition in (CG) trivial, i.e., the axiom reduces to
``$R(a,z)\subseteq R(a,y)$ if and only if $z\in R(a,y)$''. Since
$R(a,z)\subseteq R(a,y)$ implies $z\in R(a,y)$ we are simply left with
axiom (B2), i.e., (CG) implies (B2). As we shall see below, however, string
crossover in general does not satisfy (B2) and thus (CG) cannot not hold in
general.  Similarly, we cannot use Lemma 1 of \cite{Changat:10}, which
states that $G_R$ is connected whenever $R$ is a transit function
satisfying (B1) and (B2).

Allowing conditions not only on $R$ but also on its closure $\widehat{R}$
we can make use of the fact that  $G_R=G_{\widehat{R}}$. Since
$\widehat{R}$ satisfies the monotonicity axion (M) by construction, (B2) is
also satisfied. Thus $R_G$ is connected if at least one of the following
two conditions is satisfied: (i) $\widehat{R}$ satisfies (B1), or
(ii)  $\widehat{R}$ satisfies
\begin{description}
\item[(CG')] $x\in \widehat{R}(a,z)$ and $z\in\widehat{R}(a,y)$ if and only
  if $z\in \widehat{R}(x,y)$
\end{description}
The latter is equivalent to (CG) whenever $R$ satisfies (M). To see this
observe that $R(a,x) \subseteq R(a,y)$ implies $x\in R(a,y)$ and by (M)
$R(x,y)\subseteq R(a,y)$.

So far, we lack a condition for the connectedness of $G_R$ that can be
expressed by first order logic in terms of $R$ alone.

\section{Basic Properties of $k$-Point Crossover}
\label{sect:basic}

We first show that the underlying graphs of $k$-point crossover transit
functions are Hamming graphs.
\begin{lemma}
\label{lem:GRk}
  $G_{R_k} = \prod_{i=1}^n K_{a_i}$ for all $1\le k\le n-1$.
\end{lemma}
\begin{proof}
  Since $R_j(x,y)\subseteq R_k(x,y)$ for $j\le k$ by definition, it
  suffices to consider $R_1$. By definition, $R_1(x,y)=\{x,y\}$ if and only
  if $x$ and $y$ differ in a single coordinate, i.e., for which $d(x,y)=1$,
  i.e., $x$ and $y$ are adjacent in $\prod_{i=1}^n K_{a_i}$.  Obviously,
  $R_k(x,y)=R_1(x,y)$ in this case.  If there are two or more sequence
  positions that are different between the parents, then the crossover
  operator can ``cut'' between them to produce and generate an off-spring
  different from either parent so that $|R_1(x,y)|>2$.
  \
\end{proof}

From Lemma \ref{lem:GRk} and Corollary \ref{cor:A1-A4} we immediately
conclude that the $k$-point crossover transit function $R_k$ satisfies
(A1), (A2'), (A3), and (A4).

\begin{lemma}
  \label{lem:hatR-I}
  Let $R_k$ be the $k$-point crossover function. Then $\widehat{R_k}(x,y) =
  I_{G_{R_k}}(x,y)$ for all $x,y\in X$ and all $k\ge 1$.
\end{lemma}
\begin{proof}
  By construction $z\in \widehat{R_k}(x,y)$ agrees in each position with at
  least one of the parents, i.e., $z_i\in\{x_i,y_i\}$ for $1\le i\le n$,
  and thus $\widehat{R_k}(x,y) \subseteq I_{G_R}(x,y)$. Conversely, choose
  an arbitrary $z\in I_{G}(x,y)$. Find the first position $k$ in the
  coordinate representation in which $z$ disagrees with $x$ and form the
  recombinant $y'\in R_1(x,y)$ that agrees with $x$ for $i<k$ and with $y$
  for all $i\ge k$.  Then form the $x'\in R_1(x,y')\subseteq
  \widehat{R_1}(x,y)$ by recombining again after position $k$. By
  construction, $x'$ agrees with $z$ at least for all $i\le k$, i.e., in at
  least one position more than $x$. Since $x'\in \widehat{R_1}(x,y)$ we can
  repeat the argument at most $n$ time to find a sequence $x^{(n)}\in
  \widehat{R_1}(x,y)$ that agrees with $z$ in all positions. Since
  $R_1(x,y)\subseteq R_k(x,y)$ for all $k\ge 1$, we conclude that $z\in
  \widehat{R_k}(x,y)$.
  \
\end{proof}
As an immediate corollary we have:
\begin{corollary}
  $k$-point crossover is MP-geometric for all $k\ge 1$.
\end{corollary}
MP-geometricity is a desirable property for crossover operators in general
because it ensures that repeated application eventually produces the entire
geodesic interval of the underlying graph structure.

Lemma~\ref{lem:hatR-I} also implies a negative answer to one of the
questions posed by \cite{Mulder:08}: ``Is the geodesic convexity uniquely
determined by the geodesic interval function $I(u,v)$ of a connected
graph?''. More precisely, Lemma~\ref{lem:hatR-I} shows that the $k$-point
crossover transit function $R_k$ also generates the geodesic convexity and
hence that the geodesic convexity is not uniquely determined by the
interval function $I$ as the $I_{G_{R_k}}(x,y)$, being the interval in a
hypercube, is itself convex.

A trivial consequence of Lemma~\ref{lem:hatR-I}, furthermore, is the well
known fact that the transit function of uniform crossover $R_{n-1}$ is the
interval function on the Hamming graph:
\begin{corollary}
  $R_{n-1}(x,y)=\widehat{R_{n-1}}(x,y) = I_{G_R}(x,y)$ for all $x,y\in X$.
\end{corollary}
For small distances, $k$-point crossover also produces the full geodesic
interval in a single step:
\begin{lemma}\label{lem:fullinterval} 
  $R_k(x,y)=I_{G_R}(x,y)$ if and only if $d(x,y)\le k+1$.
\end{lemma}
\begin{proof}
  If $d(x,y)\le k+1$ we can place one cross-over cut between any two
  position at which $x$ and $y$ differ. In this way we obtain all possible
  recombinations, i.e., $R_k(x,y)$ is a subhypercube of dimension $k+1$ in
  the underlying graph $G_R$. Conversely if  $d(x,y) > k+1$ then there is
  $z \in I_{G_R}(x,y)$ such that $z$ requires more than $k$ cross-over
  cuts between positions at which $x$ and $y$, which completes the proof.
  \
\end{proof}

Next, we observe that the transit sets of $k$-point crossover can be
constructed recursively.
\begin{theorem}\label{thm:recurse}
$\displaystyle
R_k(u,v)= \bigcup_{z \in R_{k-1}(u,v)} \left[R_1(u,z) \cup R_1(z,v)\right]$.
\end{theorem}
\begin{proof}
  W.l.o.g we can assume that $u= 0 \ldots 0$ and $v= 1 \ldots 1$.  Let $a \in
  R_k(u,v)$ and without loss of generality we can assume that $a$ ends with
  0. Let $a_j$ denote the coordinate, with the last appearance of 1 in
  $a$. Let $b$ be an element with $b_i=a_i$ for $1 \leq i \leq j$ and
  $b_i=1$ otherwise. It follows that $b \in R_{k-1}(u,v)$ and moreover $a
  \in R_1(b,v)$.
  \
\end{proof}

A key property in the theory of transit functions is the so-called Pasch
axiom
\begin{description}
\item[(Pa)] For $p, a, b\in X$, $a'\in R(p,a)$ and $b'\in R(p,b)$ implies
  that $R(a',b)\cap R(b',a) \neq \emptyset$.
\end{description}

\begin{figure}
\label{fig:Pasch}
  \begin{center}
    \includegraphics[width=\textwidth]{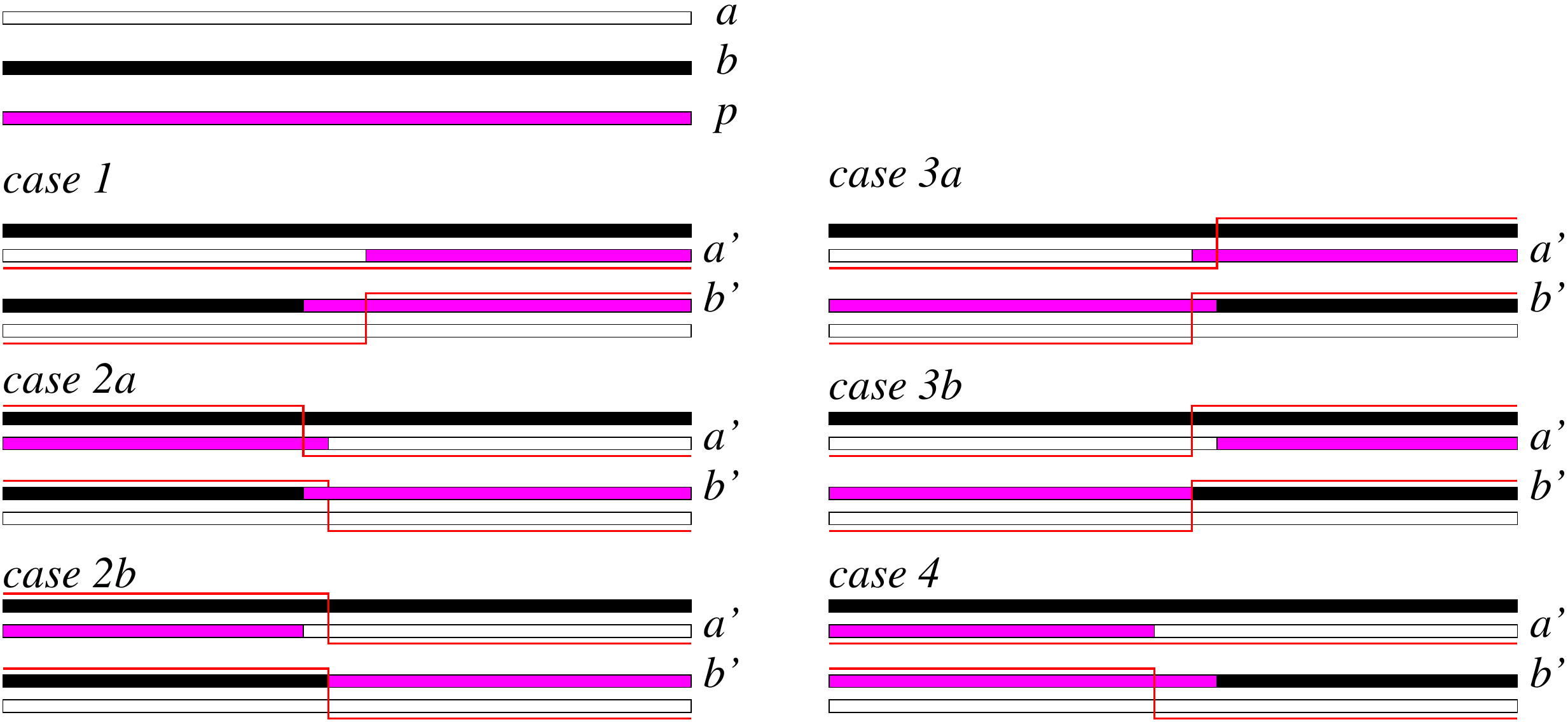}
  \end{center}
  \caption{Sketch of the proof of Lemma~\ref{lem:Pasch}. We distinguish 6
    cases depending on how $a$' and $b'$ are constructed in $R_1(a,p)$ and
    $R_1(b,p)$, respectively. The red lines indicate the explicit
    construction of an element in $R_1(a,b')\cap R_1(a',b)$.}
\end{figure}

\begin{lemma}
\label{lem:Pasch}
$R_1$ satisfies the Pasch axiom {\rm(Pa)}.
\end{lemma}
\begin{proof}
  Consider three arbitrary strings $a$, $b$, and $p$. Then $a'\in R_1(a,p)$
  is a concatenation of a prefix of $a$ with the corresponding suffix of
  $p$, or \emph{vice versa}. Each $b'\in R_1(b,p)$ has an analogous
  representation, leading to four cases depending on whether $p$ is a prefix
  or a suffix of $a'$ and $b'$, resp., see Fig.~\ref{fig:Pasch}. In case 1,
  $a'\in R_1(b',a)$ if $a'$ has a shorter $p$-suffix than $b'$. Otherwise
  $b'\in R_1(a',b)$.  In case 2, $a'$ has a $p$-prefix up to $k$ and $b'$
  has a $p$-suffix starting at $l$. If the two parts of $p$ overlap, i.e.,
  $l\le k$ then $(b_1\dots b_l,p_{l+1}\dots p_{k},a_{k+1}\dots a_n)\in
  R_1(b,a')\cap R_1(a,b')$. If $k<l$ then a common crossover product is
  obtained by recombining both $b$ with $a'$ and $a$ with $b'$ at position
  $k$. Case 3, $a'$ has a $p$-suffix and $b'$ has an $p$-prefix, can be
  treated analogously. Case 4, in which $p$ matches a prefix of both $a'$
  and $b'$ can be treated as in case 1. In summary, thus $R_1(a',b)\cap
  R_1(a,b')\ne \emptyset$ for any choice of $a'\in R_1(a,p)$ and $b'\in
  R_1(b,p)$, i.e., $R_1$ satisfies (Pa).
  \
\end{proof}

\begin{theorem}
  $R_k$ satisfies axiom {\rm(Pa)} for all $k\ge1$.
\end{theorem}
\begin{proof}
For fixed $a,b,p$ let $a'\in R_k(a,p)$ and $b'\in R_k(b,p)$.
By Theorem~\ref{thm:recurse} we have
\begin{equation*}
 R_k(a',b) = \bigcup_{z\in R_{k-1}(a',b)}[ R_1(a',z)\cup R_1(z,b) ] \qquad
   R_k(a,b') = \bigcup_{y\in R_{k-1}(a,b')}[ R_1(b',y)\cup R_1(y,a) ]
\end{equation*}
and hence
\begin{equation*}
\begin{split}
&R_k(a',b)\cap R_k(a,b') = \\
&\left( \bigcup_{z\in R_{k-1}(a',b)}[ R_1(a',z)\cup R_1(z,b) ] \right) \cap
\left( \bigcup_{y\in R_{k-1}(a,b')}[ R_1(b',y)\cup R_1(y,a) ] \right)
 = \\
&
\left(
\left( \bigcup_{z\in R_{k-1}(a',b)}[ R_1(a',z)\cup R_1(z,b) ] \right) \cap
\left( \bigcup_{y\in R_{k-1}(a,b')} R_1(b',y) \right)
\right)
   \cup\\
&\left(
\left( \bigcup_{z\in R_{k-1}(a',b)}[ R_1(a',z)\cup R_1(z,b) ] \right) \cap
\left( \bigcup_{y\in R_{k-1}(a,b')} R_1(y,a)\right)
\right)\\
&
\supseteq
\bigcup_{{z\in R_{k-1}(a',b) \atop y\in R_k-1(b',a)}} [ R_1(a',z)\cap R_1(y,a)]
\end{split}
\end{equation*}
Since $z=b\in R_{k-1}(a',b)$ and $y=b'\in R_{k-1}(a,b')$ we conclude
$R_k(a',b)\cap R_k(a',b) \supseteq R_1(a',b)\cap R_1(a,b') \ne \emptyset$
by Lemma~\ref{lem:Pasch}.  
\
\end{proof}

The Pasch axiom (Pa) implies in particular (B3), as shown by
\cite{vandeVel:93}. Lemma 1 of \cite{MulderNebesky:09} therefore implies
that $R_k$ also satisfies
\begin{description}
  \item[(C4)] $z\in R(x,y)$ implies $R(x,z)\cap R(z,y)=\{z\}$,
\end{description}
which in turn implies (B1).

Furthermore, $\widehat{R}$ also satisfies (M) and therefore in particular
(B2). As an immediate consequence we conclude that $\widehat{R_k}$ is
geometric in the sense of Nebesk{\'y}. Note that this is not true for $R_k$
itself since (B2) is violated for all $k<n-1$ for all pairs of vertices
with distance $d(x,\bar x)=n$. Lemma 1 of \cite{Changat:10}, furthermore,
implies that $G_{\widehat{R_k}}=G_{R_k}$ is connected since $\widehat{R_k}$
is a transit function satisfying (B1) and (B2).

The requirement that $G$ is connected in Corollary~\ref{cor:hypercube} and
\ref{cor:A1-A4} can therefore be replaced also by requiring that
$\widehat{R}$ satisfies (Pa).

The main result of \cite{Nebesky:94}, see also \cite{MulderNebesky:09},
states that a geometric transit function $R$ equals the interval function
of its underlying graph, $R=I_{G_R}$, if and only if $R$ satisfies in
addition the two axioms
\begin{description}
  \item[(S1)]
    $|R(x,y)|=|R(z,w)|=2$,
    $x \in R(y,w)$, and
    $y,w\in R(x,z)$,
    implies $z \in R(y,w)$.
  \item[(S2)]
    $|R(x,y)|=|R(y,w)|=2$, $y
    \in R(x,y)$, $ w \notin R(x,z)$, $z \notin R(y,w)$ implies $y \in R(x,w)$.
\end{description}
Again we need (S1) and (S2) to hold for $\widehat{R}$ rather than $R$
itself.

\begin{lemma}\label{s1}
The 1-point crossover operator $R_{1}$ satisfies the {\rm(S1)} axiom.
\end{lemma}
\begin{proof} Let $R_{1}$ be 1-point crossover operator.  Since
  $|R(u,x)|=|R(v,y)|=2$, it follows that $u$ and $x$ as well as $v$ and $y$
  differ in only a single coordinate. Writing
  $u=(u_{1}, u_{2}, \ldots, u_{i}, u_{i+1},\ldots,u_{n})$ and 
  assuming $u\in R(x,y)$ we must have either
  \begin{description}
    \item[(1)]  $u=(u_{1},u_{2},\ldots,u_{i},u_{i+1},\ldots,u_{n}) = 
                   (x_{1},x_{2},\ldots,x_{i},y_{i+1},\ldots,y_{n})$, or 
    \item[(2)]  $u=(u_{1},u_{2},\ldots,u_{i},u_{i+1},\ldots,u_{n}) =
                   (y_{1},y_{2},\ldots,y_{i},x_{i+1},\ldots,x_{n})$.
  \end{description}
  W.l.o.g., suppose $u$ is of the form (1), therefore 
  $u_{1}, u_{2}, \ldots,u_{i}=x_{1},x_{2},\ldots,x_{i}$ and 
  $u_{i+1},\ldots,u_{n}=y_{i+1},\ldots,y_{n}$. Let $x\in R(u,v)$. 
  Since $u$ is of the form (1), we have 
  $u_{1}, u_{2}, \ldots,u_{i}=x_{1},x_{2},\ldots,x_{i}$ and 
  $x_{i+1},\ldots,x_{n}=v_{i+1},\ldots,v_{n}$. Let $y\in R(u,v)$. 
  since $u$ is of the form (1), we have 
  $y_{1},y_{2},\ldots,y_{i}=v_{1},v_{2},\ldots,v_{i}$ and 
  $y_{i+1},\ldots,y_{n}=u_{i+1},\ldots,u_{n}$. 
  Hence $v=(v_{1},v_{2},\ldots,v_{i},v_{i+1},\ldots,v_{n})$ 
  can be written as\\ 
  $v=(y_{1},y_{2}, \ldots, y_{i},x_{i+1},\ldots,x_{n})$, which 
  implies $v\in R(x, y)$. Thus the axiom (S1) follows.
  \
\end{proof}

\begin{lemma}\label{S2}
The 1-point crossover operator $R_{1}$ satisfies axiom {\rm(S2)}.
\end{lemma}
\begin{proof}
From $|R(u,x)|=|R(v,y)|=2$, it follows that $u$ and $x$ differ 
in only one coordinate, say $i$, and $v$ and $y$ differ in a single
coordinate, say $j$. W.l.o.g., let $i\leq j$. 
Since $v\notin R(x,y)$, $y$ and $v$ differ only in position $j$, 
we conclude that
\begin{description}
  \item[(*)]  $x_{j},\ldots,x_{n}\neq v_{j},\ldots,v_{n}$.
\end{description}
From $x\in R(u,v)$ and (*) we obtain 
$x_{1},\ldots,x_{j-1}=v_{1},\ldots,v_{j-1}=y_{1},\ldots, y_{i-1}$. 
Hence $x_{j},\ldots,x_{n}=u_{j},\ldots, u_{n}$. 
Therefore $x=(x_{1},\ldots, x_{j-1}, x_{j},\ldots,x_{n}) =
(y_{1},\ldots, y_{j-1}, u_{j},\ldots, u_{n})$. 
This implies $x\in R(u,y)$ and axiom (S2) follows.
\
\end{proof} 

As shown in \cite{MulderNebesky:09}, the axiom
\begin{description}
\item[(MO)] $R(x,y) \cap R(y,z) \cap R(z,x) \ne \emptyset$
\end{description}
implies both (S1) and (S2).

On hypercubes, i.e., assuming an alphabet with just two letters, uniform
crossover $R=\widehat{R_k}$ satisfies $|R(x,y) \cap R(y,z) \cap
R(z,x)|=1$. The unique \emph{median} $m=R(x,y) \cap R(y,z) \cap R(z,x)$ is
defined coordinate-wise by majority voting of $x_i,y_i,z_i\in\{0,1\}$, see
\cite{Mulder:80}. On hypercubes, $\widehat{R_k}$ thus satisfy (MO).
This argument fails, however, for general Hamming graphs. The reason is
that axiom (MO) fails for each position at which the three sequences
$x,y,z$ are pairwise distinct: $\{0,1\}\cap\{1,2\}\cap\{2,0\}=\emptyset$.

For $z \in R_k(x,y)$ let $I$ denote the set of indices $0=i_0\le i_1\le i_2
\le i_k=n$ from definition \ref{def:xover} such that $z$ is a $k$-point
crossover offspring of $x$ and $y$. If $z$ is an offspring such that $x$ is
placed before $y$ in the definition we denote this by $z=x\times_{I}y$ and
$z=y\times_{I}x$ otherwise.

\begin{lemma}\label{lemma:switch}
  Let $d(a,b) > k+1$. If $s=a \times_{I} b$, $t=a \times_{I} b$ 
  and $|I|=k$, then\\ 
  $s \times_{j} t \notin R_k(a,b)$ and 
  $t \times_{j} s \notin R_k(a,b)$ holds for all $j \notin I$.
\end{lemma}
\begin{proof}
  Since $j \notin I$ it follows that 
  $s \times_{j} t = (a \times_{I} b)\times_{j} (b \times_{I} a)$ and 
  $t \times_{j} s = (b \times_{I} a)\times_{j} (a \times_{I} b)$, 
  we have $s \times_{j} t, t \times_{j} j 
  \in R_{k+1}(a,b) \setminus R_{k}(a,b)$.
  \
\end{proof}

\cite{Gitchoff:96} conjectured that for each transit set $R_k(x,y)$
there is a unique pair of parents from which it is generated unless
$R_k(x,y)$ is a hypercube. We settle this conjecture affirmatively:

\begin{theorem}\label{uniquetransitsets}
If $d(u,v),d(x,y)>k+1$ then
$R_k(u,v) = R_k(x,y)$ if and only if $\{u,v\} = \{x,y\}$.
\end{theorem}
\begin{proof} The implication from right to left is trivial. 
  For other direction we use Lemma \ref{lemma:switch}. 
  Assume, for contradiction, that $R_k(u,v) = R_k(x,y)$ and 
  $\{u,v\} \neq \{x,y\}$. Then $x,y \in R_k(u,v)$ and $u,v \in R_k(x,y)$. 
  From $R_k(u,v) = R_k(x,y)$ it follows also that 
  $\widehat{R_k}(u,v)=\widehat{R_k}(x,y)$, which in turn implies 
  $d(u,v)=d(x,y)$. Therefore, there exists a set of indices $I$, 
  $|I| = k$, such that $x=u \times_{I} v$ and $y=v \times_{I} u$. 
  From $d(x,y)>k+1$ and Lemma \ref{lemma:switch} we conclude that 
  there exist $j \notin I$ such that $x \times_{j} y \notin R_k(u,v)$. 
  Hence $R_k(u,v) \neq R_k(x,y)$. This contradiction completes 
  the proof of the theorem.
  \
\end{proof} 

For the special case $k=1$, Theorem \ref{uniquetransitsets} for $k=1$
implies the following statement.
\begin{description}
\item[(H3)] For every $x,y,u,v \in X$, $u\neq v$, $x \neq y$, $|R(x,y)|>4$,
  $R(u,v) \subseteq R(x,y)$ implies that either $R(u,v)=\{u,v\}$ or
  $\{u,v\}=\{x,y\}$.
\end{description}
For $x,y\in X$ with $d(x,y)=t \geq 3$, the transit set $R_1(x,y)$ induces a
cycle of size $2t$, an hence the only other transit sets that are included
in $R_1(x,y)$ are singletons and edges. 

%%%%%%%%%%%%%%%%%%%%%%%%%%%%%%%%%%%%%
\section{Graph theoretical approach for $k$-point crossover operators}
\label{sect:GTk}
%%%%%%%%%%%%%%%%%%%%%%%%%%%%%%%%%%%%%

Transit sets $R(x,y)$ inherit a natural graph structure as an induced
subgraph of the underlying graph $G_R$. In the case of crossover operators
and their corresponding transit sets $R_k(x,y)$, the distance in the
underlying graph plays a crucial role in their characterization.

Recall that $n$-dimensional hypercubes are antipodal graphs, i.e., for any
vertex $v$ there is a unique antipodal vertex $\overline{v}$ with
$d(v,\overline{v})= \text{diam} (G) = n$, where $\text{diam} (G)$ denotes the diameter of graph $G$. The vertex $\overline{v}$ is
obtained from $v$ by reversing all coordinates.
    
\begin{theorem}\label{thm:antipodalpc}
  $R_k(x,y)$ induces an antipodal graph such that 
  $\overline{x}=y$, $\overline{y}=x$, and for each $u \in
  R_k(x,y)$ of the form $u=x \times_{I} y$ we have 
  $\overline{u}=y \times_{I} x$.
\end{theorem}
\begin{proof} 
  The definition of $R_k$ immediately implies that $x$ and $y$ are at the
  maximal distance from each other and every other $v \in R_k(x,y)$, $v=x
  \times_{I} y$, has a unique vertex at maximal distance in $R_k(x,y)$,
  that is $\overline{v}=y \times_{I} x$.
  \
\end{proof}

Note that here $v$ and $\overline{v}$ are antipodal in a subgraph
$R_k(x,y)$ and will not be antipodal in the underlying graph $G_R$, unless
$d(v,\overline{v})= \text{diam}(G_R)$.  This is not the only property inherited
from hypercubes. We say that $H$ is an \emph{isometric subgraph} of a graph
$G$ if for every pair of vertices $u,v \in V(H)$ the distance from $G$ is
preserved, i.e., if $d_H(u,v)=d_G(u,v)$. Isometric subgraphs of hypercubes
are known as \emph{partial cubes} \cite{handbookproductgraphs:2011,
  ovchinnikov2011cubes}.  \cite{Gitchoff:96} showed (1) that $R_1(x,y)$
induces $C_{2t}$, a cycle of length $2t$, where $t=d(x,y)$, and (2) that
$d_{C_{2t}}(u,v)=d_{G_{R_1}}(u,v)$ holds for every pair $u,v \in
R_1(x,y)$. In other words $R_1(x,y)$ is a partial cube. Theorem
\ref{thm:recurse} implies that this result holds in general:
\begin{corollary}\label{cor:partialcube1}
The $k$-point crossover operator $R_k$ induces a partial cube.
\end{corollary}
In particular, therefore, $R_k$ always induces a connected subgraph of $G_R$. 

\smallskip\par\noindent
\emph{\textbf{In the remainder of this section we consider only the binary case.}}

\begin{definition} 
  Let $R$ be a transit function $R$ on a set $X$. Then we set $uv \parallel
  xy$ if and only if $v,x \in R(u,y)$ and $u,y \in R(v,x)$.
\end{definition} 
The binary relation $\parallel$ was introduced by \cite{dress2007x-nets}
in the context of a characterisation of so called $X$-nets, a structure
from phylogenetic combinatorics that is intimately connected with partial
cubes. Indeed, \cite{dress2007x-nets} showed that $\parallel$ can be used
to characterize partial cubes:
\begin{proposition}[\cite{dress2007x-nets}]\label{thm:X-nets}
  Let $G$ be a graph and $R=I_G$, then $G$ is a partial cube if and only if
  the relation $\parallel$ is an equivalence relation on the set of its
  edges.
\end{proposition}
By the definition, the relation $\parallel$ is reflexive and
symmetric. Therefore it suffices to require that $\parallel$ is a
transitive relation. Proposition \ref{thm:X-nets} thus can be translated
into the language of transit functions:
\begin{theorem}
\label{thm:partialcubes}
Let $R$ be a transit function on a set $X$. Then the underlying graph $G_R$
is partial cube if and only if $R$ satisfies:
\begin{description}
\item[{\rm(AX)}] for all $a,b,c,d,e,f \in X$, with
  $|R(a,b)|=|R(c,d)|=|R(e,f)|=2$ and $ab \parallel cd$ and $cd \parallel
  ef$ it follows that $ab \parallel ef$ .
\end{description}
\end{theorem}
It is worth noting that the axiom (AX) can be also described purely in a
transit sets notation as follows:
\begin{description}
\item[(AX')] for all $a,b,c,d,e,f \in X$, with
  $|R(a,b)|=|R(c,d)|=|R(e,f)|=2$ and $b,c \in R(a,d)$, $a,d \in R(b,c)$,
  $d,e \in R(c,f)$ and $c,f \in R(d,e)$ it follows that $b,e \in R(a,f)$
  and $a,f \in R(b,e)$.
\end{description}

\begin{figure}[htbp]
\begin{center}
\includegraphics[width=0.45\textwidth]{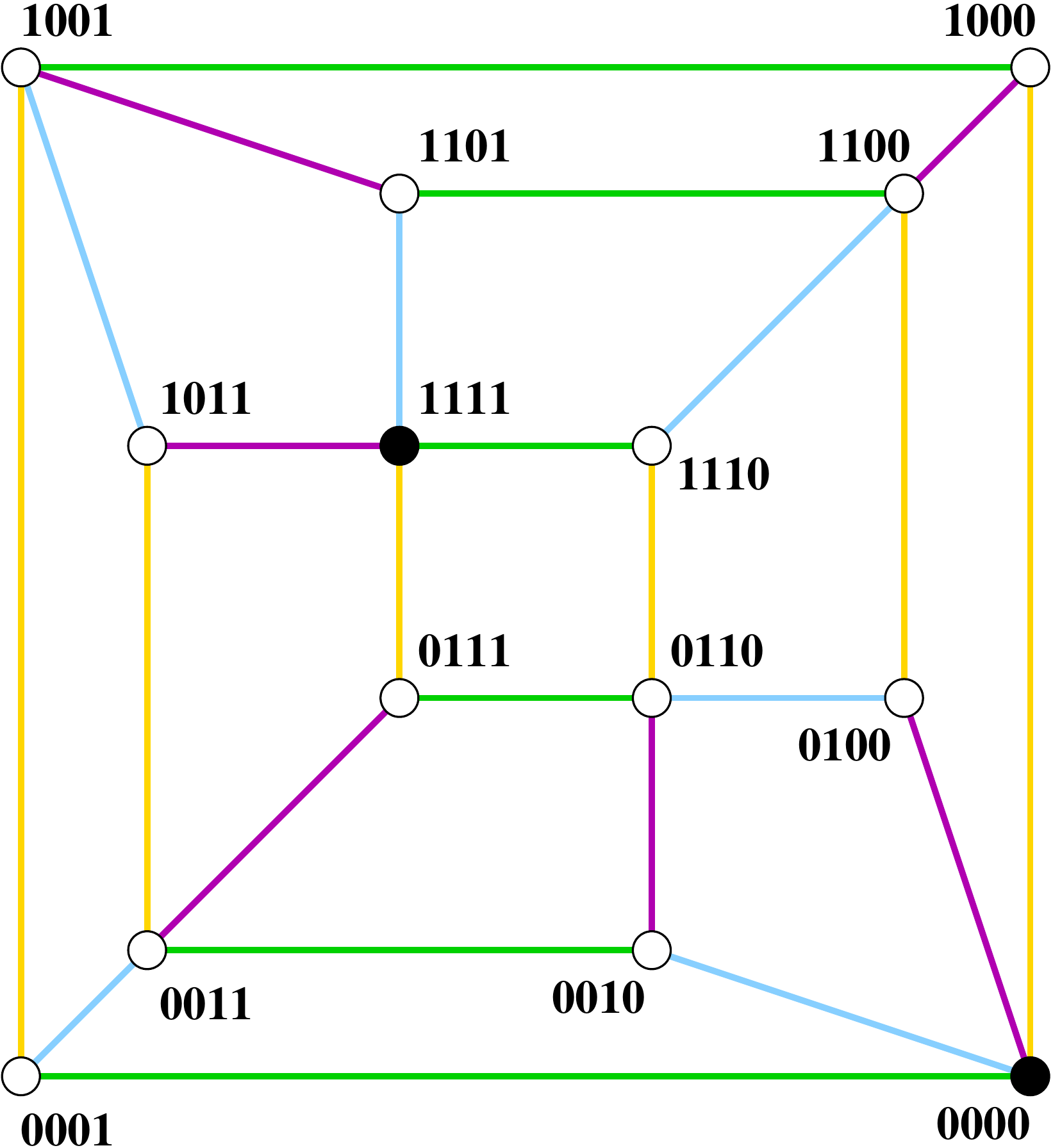}
\caption{$R_2(0000,1111)$ together with colored cuts.}
\label{fig:R24}
\end{center}
\end{figure}

For a partial cube $G$, the equivalence classes of the relation $\parallel$
are called \emph{cuts} and we denote the set of all cuts by
$\mathcal{C}=\{C_1,C_2,\ldots, C_n\}$, where $n$ is the dimension of the
smallest hypercube into which $G$ embeds isometrically.  Cuts form a
minimal edge partition of the edge set in a partial cube with the property
that removal of all edges from a given cut results in a disconnected graph
with exactly two connected components. These are called \emph{splits}
\cite{dress2012phylogenetic,ovchinnikov2011cubes}.

Cuts of the partial cubes correspond to the coordinates in the
corresponding isometric embedding into the hypercube and they induce a
binary labelling of the strings: for a cut $C_i$ vertices from one part of
the split induced by $C_i$ are labeled ``0'' in coordinate $i$, and
vertices from the other part of the split are labeled ``1'' in coordinate
$i$. For a any pair of parallel edges $xy$, $uv$ in a partial cube the
notation can be chosen such that $d(u,x) = d(v,y) = d(u,y)-1 =
d(v,x)-1$. The distance between any two vertices in a partial cube
therefore can be computed as the Hamming distance between the corresponding
binary labelings, which in turn correspond to the number of cuts that
separate the two vertices. In other words, any shortest path between two
vertices in a partial cube is determined by the cuts it
traverses. Moreover, any shortest path traverses each cut at most once. We
refer to \cite{handbookproductgraphs:2011, ovchinnikov2011cubes} for the
details; there, the cuts are called $\Theta$-classes.

Let us denote the cuts appearing in the partial cube $R_1(x,y)$ by
$C(x,y)$. We have $|C(x,y)|=d(x,y)$. For any pair of vertices $x,y$ in a
hypercube with $d(x,y)=t$ we have $t!$ possible ways to choose a shortest
path between them, because each of the $t!$ possible orders in which the
corresponding cuts that are traversed results in a distinct path. Therefore
there are also $t!$ ways to choose an isometric cycle through $x$ and
$y$. The definition of the $1$-point crossover operator, on the other hand,
identifies a unique isometric cycle between $x, y \in V(G_R)$.

The binary labelling of vertices in a partial cube naturally induces a
lexicographic ordering of vertices. Similarly, by taking first the
labelling of the minimal vertex and concatenating it with the labelling of
the remaining vertex, we can also lexicographically order the edges of a
partial cube. The idea can further be generalized to a lexicographic ordering of
all paths and cuts of a partial cube. The following result shows the
$1$-point crossover is intimately related to this lexicographic order.

\begin{theorem}
  Let $x,y \in X=\{0,1\}^n$. Then $R_1(x,y)$ consist of all vertices
  appearing on lexicographically minimal and maximal paths between $x$ and
  $y$.
\end{theorem}
\begin{proof}
  The statement follows immediately from the definition of the 1-point
  crossover operator.
  \
\end{proof}

\begin{problem}
  Is it true that $R_k(x,y)$ consist of all vertices appearing on
  $\frac{k!}{2}$ pairs of first and last lexicographically minimal and
  maximal paths between $x$ and $y$?
\end{problem}

\begin{figure}[h!]
 \label{fig:antilexicographic}
%\vspace{-0.8cm}
\centering
\includegraphics[width=0.40\textwidth]{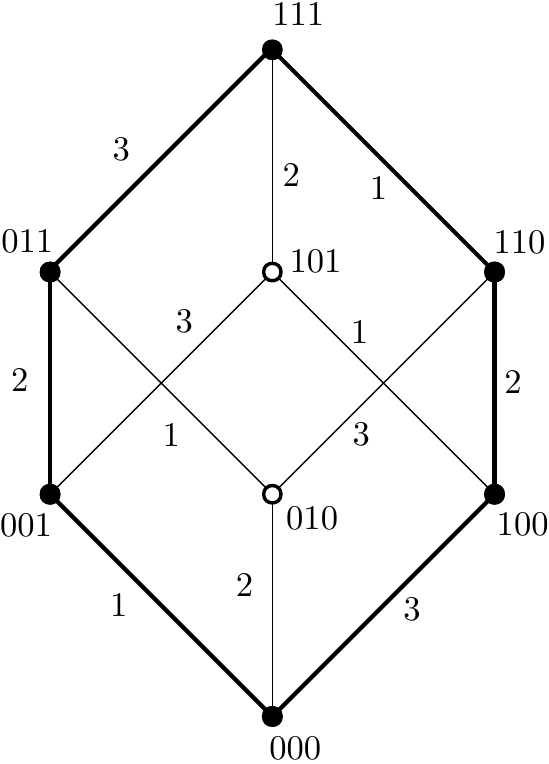}
\caption{$R_1(000,111)$ in $K_2^3$ is represented with black vertices and
  fat edges. Lexicographic ordering of the vertices of $K_2^3$: starting at
  000 bottom-up and then in each level from left to right; labelling of the
  edges corresponds to the induced lexicographic ordering of cuts.}
\end{figure}

For $x,y \in X\{0,1\}^n$ and any shortest path between them, there is
exactly one path along which the cuts appear in the reverse order. Consider
any $u\in R_1(x,y)\setminus\{x,y\}$. There is is exactly one shortest path
between $u$ and $x$ in $R_1(x,y)$.  For $k>1$ and $d(x,y)=t$ both $x$ and
$y$ have exactly $t$ neighbours in $R_k(x,y)$. Moreover as shown above, see
example \ref{fig:R24}, for $u \in R_k(x,y)$, it may be the case that
$R_k(x,u)=\widehat{R_k}(x,u) \subseteq R_k(x,y)$. Hence the lexicograpic
order of cuts does not uniquely determine a shortest path in
$R_k(u,x)$. The structure of $R_k(u,v)$ hence is much richer and calls for
more ``dimensions''. We explore this structure in more details in the
sections \ref{sect:topological} below.

\begin{problem}
  Compute the size of cuts for $R_k$, i.e., the number of edges belonging
  to the common cut.
\end{problem}

The degree sequence of the graphs induced by $1$-point crossover and
uniform crossover operators are monotone. In the first case all values are
equal to $2$, and in the second case they equal the length of the string
$n$.

\begin{problem}
\label{pro:degreeseq}
Let $d(a,b)=t > k+1 > 2$. Determine degree sequences of the graphs induced
by $R_k(a,b)$.
\end{problem}
We will solve this problem completely for the special case of 2-point
crossover operators in Section \ref{sect:topological}.

\begin{lemma}
  \label{lem:minmax}
  Let $a,b \in X=\{0,1\}^n$ and $k>1$. Then the maximum and minimum degree
  of a graph induced by $R_k(a,b)$ equal $n$ and $k+1$, respectively.
\end{lemma}
\begin{proof}
  Clearly the graph induced by $R_2(a,b)$ includes all neighbours of $x$
  and $y$ in $\{0,1\}^n$, hence the maximum degree of a graph induced by
  $R_2(a,b)$, and consequently of graphs induced by $R_k(a,b)$, for $k>2$,
  is $n$.  W.l.o.g., let $a=0\ldots0$ and $b=1\ldots1$. Then $R_k(a,b)$
  consist of all binary strings with less than $k+1$ blocks of consecutive
  0's or 1's. Hence the minimum degree in a graph induced by $R_k(a,b)$ is
  attained by vertex corresponding to a binary string consisting of exactly
  $k$ different blocks of consecutive 0's or 1's, and they have exactly
  $k+1$ neighbours in a graph induced by $R_k(a,b)$.
  \
\end{proof}
A solution of Problem \ref{pro:degreeseq} could help solve 
\begin{problem}
  Does $R_k$ induce a $k$-connected graph?
\end{problem}

For an axiomatic characterization of $R_1$ in terms of transit functions
axioms it is easy to translate graph theoretic properties related to the
fact that $R_1(x,y)$ induces an isometric cycle in $\{0,1\}^n$ in the
language of transit functions. In addition, however, it would also be
necessary to express a consistent ordering of the cuts that appear in the
isometric cycles in terms of transit function. While this appears possible,
it seems to be cumbersome and does not promise additional insights into the
structure of the transit sets. Hence we do not pursue this issue further.

\section{Combinatorial Properties of Recombination Sets} 

Since the monotonicity axiom (M) fails for $k$-point crossover with
  $k<n-1$, \cite{Gitchoff:96} proposed the axiom
\begin{description}
\item{(GW3)} For all $x,y\in X$ and all $u,v\in R(x,y)$ holds
   $|R(u,v)|\le |R(x,y)|$.
\end{description}
stipulating monotonicity in size. This is proper relaxation of (M), which
obviously (M) implies (GW3). In order to derive explicit expressions for
$|R_k(x,y)|$ we note that, for given vertices $x$ and $y$, the hypercube
can be relabeled in such a way that $x$ becomes the all-zero string and $y$
is a 01-string with 1's at exactly the positions where $x$ and $y$ differ.
Thus the size of the recombination sets $|R_k(x,y)| =: r_k(t)$ depends only
on the order $k$ of the recombination operator and the Hamming distance
$t:=d(x,y)$. In the following we write 
\begin{equation} 
  \Phi_h(n):= \sum_{i=0}^h \binom{n}{i}\,. 
\end{equation}
In order to compute $r_k(t)$, we have to distinguish the case of small and
large Hamming distances.
  
\begin{theorem}
\label{thm:size}
Let $1\leq k < t$. Then 
\begin{equation}
  r_k(t) = 
  \begin{cases} 
    2^t           & \textrm{if}\quad t \le k \\  
    2 \Phi_k(t-1) & \textrm{if}\quad t > k \\
  \end{cases} 
\end{equation}
\end{theorem}
\begin{proof}
  Consider two strings $x$ and $y$.  From \cite{Gitchoff:96} we know that
  $|R_1(x,y)|=2t$. For all children of $x$ and $y$ that are obtained by
  $i$-point crossover, $1 \leq i \leq k $ with exactly $i$ cuts, we have
  $i$ possibilities for choosing the cuts along $t-1$ positions.  This
  amounts a total of $\binom{t-1}{i}$ possibilities. In 2 different choices
  for the ordering of parents. If $t\le k$ cuts may be placed
  simultaneously between any two positions in which $x$ and $y$ differ,
  i.e., $i$ takes values from $0$ to $t-1$. Thus $r_k(t)=2 \sum_{i=0}^{t-1}
  \binom{t-1}{i} = 2\Phi_{t-1}(t-1)=2\cdot 2^{t-1} =2^t$. For $t>k$ the
  number of possible cuts is limited by $k$ and hence $r_k(t)=2 \Phi_k(t-1)$.
  \
\end{proof}
Parts of this result were already observed by \cite{Gitchoff:96}. In
particular, $r_1(t)=2t$ for $t>1$, $r_2(t) = t^2-t+2$ for $t>2$, 
$r_k(k+1)=2^{k+1}$, and $r_k(k+2)=2^{k+1}-2$. 

The latter equation shows that $R_{k-1}(x,y)$ for $d(x,y)=k+1$ misses
exactly two points compared to $R_{k}(x,y)$, i.e., $R_{k}(x,y)\setminus
R_{k-1}(x,y)=\{a,b\}$. Thus we can conclude immediately that $a=x
\times_{I} y$, $b=y \times_{I} x$ with $|I|=k$.  Since every $x,y$, $x\neq
y$, has a unique $a,b$ with the above property, we obtain another simpler
proof of the Theorem \ref{uniquetransitsets}.

From $r_k(t)=2^t$ for $t\le k+1$ and the fact that $\widehat{R_k}(x,y)$ is
a hypercube $K_2^t$ of dimension $t$ for $d(x,y)=t$ we immediately conclude
that $R_k(x,y)$ is also a hypercube for $t\le k+1$.

Let $G$ be partial cube and let $H$ be a graph obtained by contracting some
of the cuts of $G$, i.e.\ by forgetting some of the coordinates in binary
labelling of vertices. If $H$ is isomorphic to some hypercube, then we say
that $H$ is a \emph{cube minor} of $G$.

\begin{lemma} 
  Consider $R_k(x,y)$ as an induced subgraph of the boolean hypercube
  $K_2^n$, and suppose $d(x,y)\ge k+1$. Then the largest cube minors of
  $R_k(x,y)$ are isomorphic to $K_2^{k+1}$.
\end{lemma} 
\begin{proof}
  This is an immediate consequence of Lemma \ref{lem:fullinterval} and 
  Theorem \ref{thm:recurse}.
  \
\end{proof} 

The Vapnik-Chervonenkis dimension (or VC-dimension) measures the
complexity of set systems. Originally introduced in learning theory by
\cite{vapnik1971dimension}, it has found numerous applications e.g.\ in
statistics, combinatorics and computational geometry, see monograph edited
by \cite{vovk2015measures}. Consider a base set $X$ and family
$\mathcal{H}\subseteq 2^X$. A set $C\subseteq X$ is \emph{shattered by
  $\mathcal{H}$} if $\{Y\cap C| Y\in\mathcal{H}\}=2^C$. The $VC$-dimension
of $\mathcal{H}$ is the largest integer $d_{VC}$ such that there is a set
$C$ of cardinality $d_{VC}$ that is shattered by $\mathcal{H}$. For
$\mathcal{H}=\emptyset$, $d_{VC}= - 1$ by
definition.

Clearly, $X$ is shattered by $\mathcal{H}=2^X$, hence the VC-dimension of
the Boolean hypercube $\{0,1\}^d$ is $d$. Now consider an even cycle
$C_{2t}$ of length $2t$, isometrically embedded into $t$ dimensional
hypercube. It is not hard to check that the VC-dimension of $C_{2t}$ is $2$
for any $t\ge 2$. More generally, the VC-dimension of a partial cube $G$,
with $d$ cuts, equals the dimension of the largest cube-minor in $G$,
because this is the largest cardinality of a set of coordinates that can be
shattered by the set of all $d$ of cuts of $G$.

\begin{theorem}
\label{prop:vc-x-over}
The VC-dimension of $R_k(x,y)$ equals $k+1$ whenever $d(x,y)>k$. Otherwise 
the VC-dimension of $R_k(x,y)$ equals $d(x,y)$. 
\end{theorem}
\begin{proof}
  If $d(x,y) \leq k$ then $R_k(x,y)$ induces graph isomorphic to
  $d$-dimensional hypercube, where $d=d(x,y)$.  Let $d=d(x,y)$. If $d>k$
  then we need to contract $d-k-1$ cuts (ignore the corresponding
  coordinates) to obtain a cube minor of dimension $k+1$.
  \
\end{proof} 

\section{Topological representation of the $k$-point crossover operators}
\label{sect:topological}

Oriented matroids \cite{OMbook1999} are an axiomatic combinatorial
abstraction of geometric and topological structures such as vector
configurations, (pseudo)hyperplane arrangements, convex polytopes, point
configurations in the Euclidean space, directed graphs, linear programs,
etc. They reflect properties such as linear dependencies, facial
relationship, convexity, duality, and have bearing on solutions of
associated optimization problems. Beyond their connection with many areas
of mathematics, the theory of oriented matroids has in recent years found
applications in diverse areas of science and technology, including
metabolic network analysis \cite{gagneur2004computation, muller2013fast,
  muller2014enzyme, reimers2014metabolic}, electronic circuits
\cite{chaiken1996oriented}, geographic information science
\cite{stell2007oriented}, and quantum gravity
\cite{brunnemann2010oriented}. 

In order to explore the relationships of $k$-point crossover operators and
oriented matroids it will be convenient to change the coordinates of the
vertices of the hypercube $K_n^2$. For the remainder of this section we
write $+$ instead of $1$ and $-$ instead of $0$ to conform with the
traditional notation in this field. Among several equivalent
axiomatizations of oriented matroids, the face or covector axioms best
captures the geometric flavour and thus is the most convenient one for our
purposes.

Let $E$ be a finite set. A \emph{signed vector} $X$ on $E$ is a 
vector $(X_e: e\in E)$ with coordinates $X_e \in \{ +, 0, - \}$. 
The support of a sign vector $X$ is the set 
$\underline{X}=\{e \in E \,|\, X_e \neq 0\}$. The
\emph{composition} $X \circ Y$ of two signed vectors $X$ and $Y$ is the
signed vector on $E$ defined by $(X \circ Y)_e = X_e$, if $X_e \neq 0$,
and $(X \circ Y)_e = Y_e$ otherwise, and their difference set is  $D(X, Y) = \{e \in E \,|\, X_e  = - Y_e \}$. We denote by $\leq$ the product
(partial) ordering on $\{-,0,+\}^E $ relative to the standard ordering 
$- < 0 < +$ of signs. 

An \emph{oriented matroid} $M$ is ordered pair $(E,\mathcal{F})$
of a finite set $E$ and a set of \emph{covectors} $\mathcal{F} \subseteq
\{+,-,0\}^{E}$ satisfying the following (face or covector) axioms: 
\begin{description}
\item[(F0)] $0=(0,0,\dots,0) \in \mathcal{F}$.
\item[(F1)] If $X \in \mathcal{F}$, then $-X \in \mathcal{F}$.
\item[(F2)] If $X, Y \in F$, then $X \circ Y \in \mathcal{F}$.
\item[(F3)] If $X, Y \in \mathcal{F}$ and  $e \in D(X,Y)$, 
      then there exists $Z \in \mathcal{F}$ such that $Z_e = 0$ and
      $Z_f=(X \circ Y)_f$ for all $f \in E \setminus D(X, Y)$.
\end{description}

A simple example for an oriented matroid are the sign vectors of a
vector subspace. More precisely, consider a subspace $V\subseteq
\mathbb{R}^{|E|}$, define, for every $v\in V$, its sign vector $s(v)$
coordinate-wise by $s_e(v)=\sgn(v_e)$ for all $e\in E$, and denote by
$\mathcal{F}$ the set of all sign vectors of $V$. Then $M=(E, \mathcal{F})$
satisfies the axioms (F0)-(F3). Oriented matroids obtained from a vector
space in this manner are called representable or linear.

The set $\mathcal{C} \subset \mathcal{F}$ of non-zero covectors that are
minimal with respect to the partial order $\leq$ are called cocircuits or
vertices of $M$. The set $\mathcal{T} \subset \mathcal{F}$ of covectors
that are maximal with respect to $\leq$ are called topes of $M$.  The set
of cocircuits determines the set of covectors: every covector $X \in
\mathcal{F} \setminus \{0\}$ has a representation of the form $X= V_1\circ
V_2 \circ \ldots \circ V_k$, where $V_1, V_2, \ldots V_k$ are cocircuits,
and $V_1, V_2,\ldots V_k \leq X$. Similarly, the topes determine the
oriented matroid: $X$ is a covector if and only if its composition with any
tope is again a tope, i.e., $\mathcal{F}= \{ X \in \{+,-,0\}^E \mid \forall
T \in \mathcal{T}: X \circ T \in \mathcal{T} \}$.

$M=(E,\mathcal{F})$ is called \emph{uniform} of rank $r$ if the support
sets of cocircuits are exactly the subsets of $E$ with $r+1$ elements.
% (if $ \mathcal{F}= \{0\}$, the rank is defined to be $|E|$). 
The big face lattice $\widehat{\mathcal{F}}$ is a lattice obtained by
adding the unique maximal element $\widehat{1}$ to the partial order $\leq$
on $\mathcal{F}$. The rank of a covector $X$ is defined as its height in
$\widehat{\mathcal{F}}$ and rank of oriented matroid is the maximal rank of
its covectors.

\begin{figure}
\begin{center}
\includegraphics[width=0.6\textwidth]{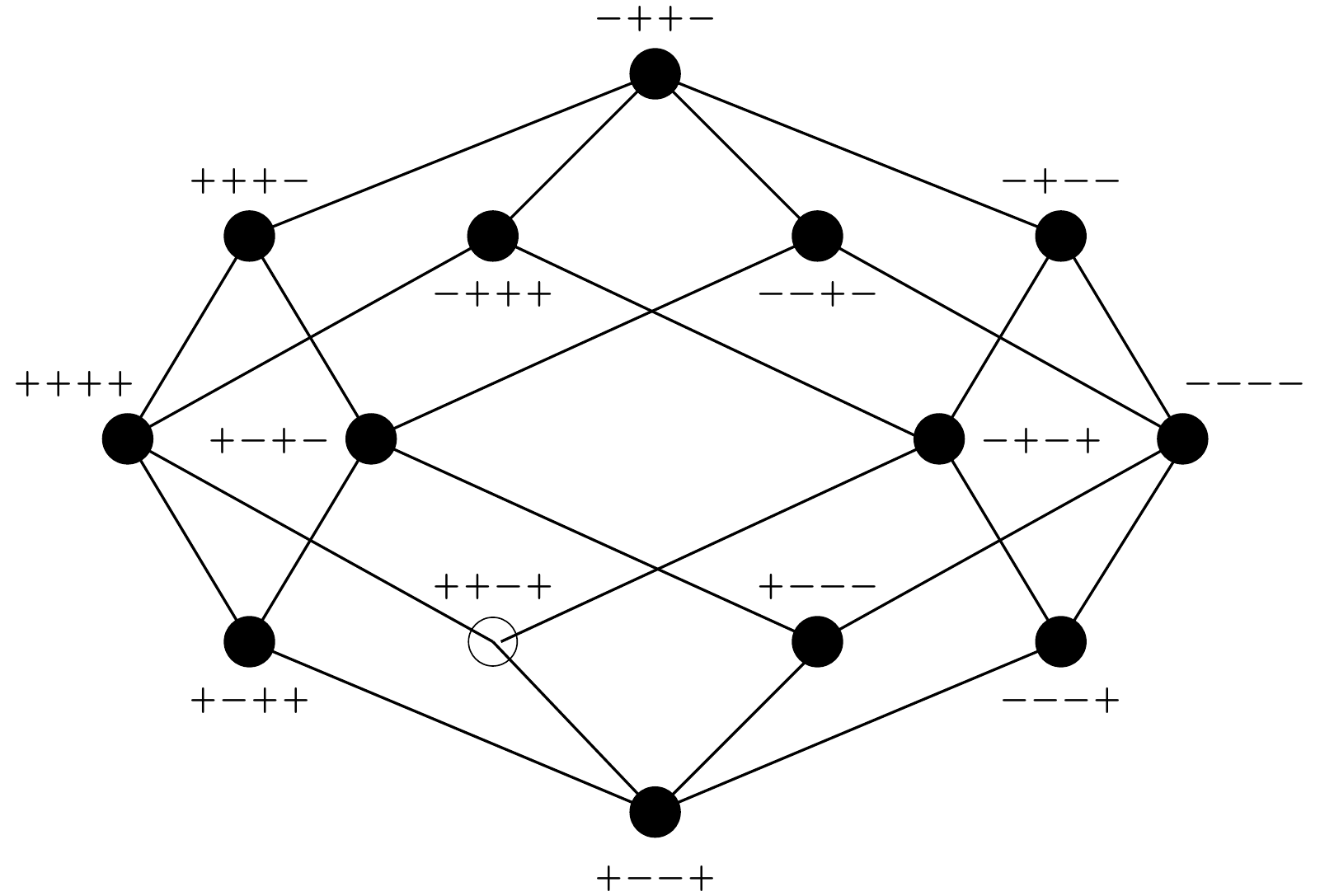}
\end{center}
\newcommand{\M}{\texttt{-}}
\renewcommand{\P}{\texttt{+}}
\caption{The rhombododecahedral graph $R_2(\M\M\M\M,\P\P\P\P)$ with
  the binary labeling corresponding to the isometric embedding into
  4-dimensional hypercube.}
\label{fig:rhombododecahedralgraph}
\end{figure} 

As an example consider $R_2(x,y)$ with $d(x,y)=5$. It can be verified that the
elements of $R_2(\verb|-----|,\verb|+++++|)$ are exactly the topes of the
oriented matroid corresponding to the Rhombododecahedron in Fig.\
\ref{fig:rhombododecahedralgraph}.  Its big face lattice is shown in Figure
\ref{fig:facelattice}.

\begin{figure}
\begin{center}
\includegraphics[width=\textwidth]{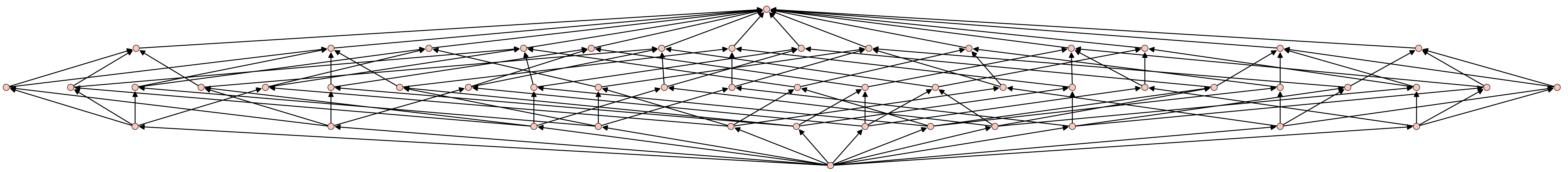}
\end{center}
\caption{Big face lattice of the rhombicdodecahedron generated using
  SageMath.}% (\verb;http://www.sagemath.org ;).}
\label{fig:facelattice}
\end{figure}

This observation can be generalized with the help of the following 
\begin{proposition}[\cite{gartner1994VCpseudo}]
\label{prop:uniformOM}
$T \subseteq \{+,-\}^X$ of VC-dimension $d$ is the set of topes of 
a uniform oriented matroid $M$ on $X$ if and only if 
$T = -T$ and $|T| = 2\Phi_{d-1}(|X|-1)$.
\end{proposition}
%https://en.wikipedia.org/wiki/Matroid_rank
%https://en.wikipedia.org/wiki/Corank
By Proposition \ref{prop:vc-x-over} and Theorems \ref{thm:size} and 
\ref{prop:uniformOM} this immediately implies
\begin{theorem}\label{thm:xover-om}
For $x,y \in  \{+,-\}^X$, with $d(x,y)=|X|=n$ the elements of 
$R_k(x,y)$ form the set of topes of a uniform oriented matroid 
$M$ on $X$ with VC-dimension $k+1$ and corank $n-k-1$.
\end{theorem}

One of the cornerstones of the theory of oriented matroids is the
Topological Representation Theorem, which connects oriented matroids with
pseudosphere arrangements, see Appendix A for detailed
definitions. Together with Theorem \ref{thm:xover-om}, it immediately
implies the following topological characterization of the recombination
sets of $k$-point crossover:
\begin{theorem}\label{thm:xover-top}
  For $x,y \in \{+,-\}^X$, with $d(x,y)=|X|=n$, the recombination set
  $R_k(x,y)$ can be topologically represented by a pseudosphere arrangement
  of dimension $k$, where the minimal elements in the big face lattice
  correspond to the intersections of exactly $k$ pseudospheres, and there
  are $2\binom{n}{k-1}$ such intersections.
\end{theorem}
The significance of this result is that it provides a representation of
crossover operators in terms of topological objects. In order to illustrate
the usefulness of Theorem \ref{thm:xover-top}, we now turn to a full
characterization of the transit graphs of 2-point crossover operators.  The
smallest non-trivial examples are the graphs $R_2(\verb|----|,\verb|++++|)$
in Fig.~\ref{fig:rhombododecahedralgraph} and
$R_2(\verb|-----|,\verb|+++++|)$ in Fig.~\ref{fig:R2-5}.

\begin{figure}
\newcommand{\M}{\texttt{-}}
\renewcommand{\P}{\texttt{+}}
\begin{center}
\includegraphics[width=0.55\textwidth]{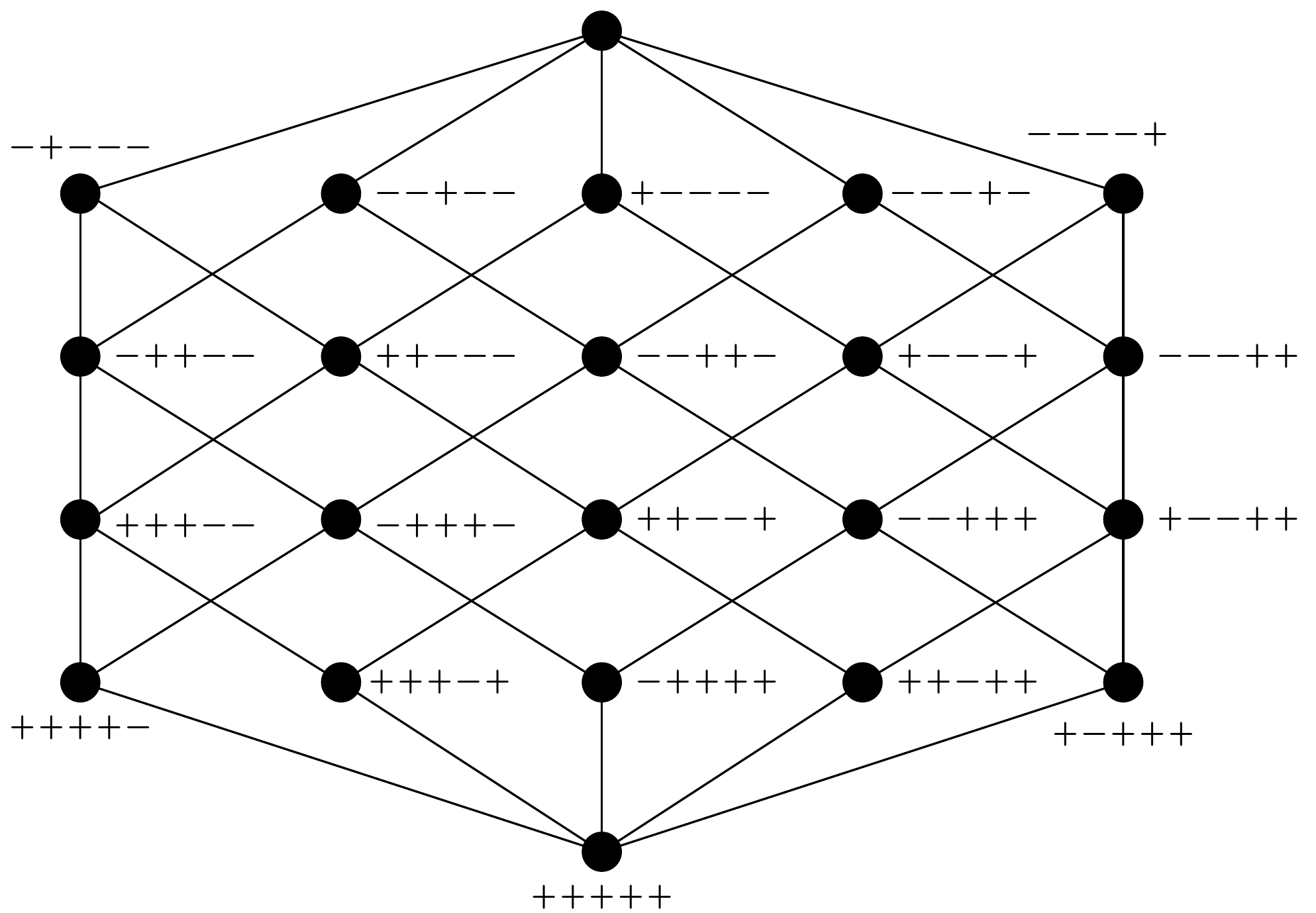}
\caption{The transit graph $R_2(\M\M\M\M\M,\P\P\P\P\P)$.}
\label{fig:R2-5}
\end{center}
\end{figure}
 
\begin{theorem}\label{thm:R2}
  $R_2(a,b)$ with $d(a,b)=t > 3$ induces antipodal planar quandrangulation,
  that is, a partial cube of diameter $t$ with $t^2-t+2$ vertices, $2t^2-2t$
  edges, $t^2-t$ quadrangles, and all cuts of size $2t-2$.
\end{theorem}
\begin{proof}
  Let $|V|$, $|E|$, $|Q|$ and $|C|$ denote number of vertices, edges,
  4-faces, and edges in a cut, respectively. From the definition of
  crossover operator, we can arbitrarily permute coordinates, hence it
  follows that each cut has the same number of edges, this justifies that
  we study $|C|$.  From Theorem \ref{thm:xover-om} it follows that vertices
  of $R_2(a,b)$ form the set of topes of uniform oriented matroid of rank
  $3$ and corank $t-3$.  As shown by \cite{fukuda1993antipodal} and in the
  book by \cite{OMbook1999}, rank $3$ oriented matroids can be represented
  by pseudocircle arrangement on $\mathbb{S}^2$. The corresponding tope
  graph is therefore planar. Hence $R_2(a,b)$ induces in particular a
  planar antipodal partial cube. Corank $t-3$ implies that each
  intersection of pseudocurves is the intersection of exactly two of
  them. Hence all faces of the dual -- the tope graph -- are 4-cycles,
  therefore $R_2(a,b)$ induces planar quadrangulation. Moreover, each
  intersection of two pseudocircles corresponds to cocircuit. In uniform
  oriented matroid of corank $t-3$ there are exactly $2 \binom{t}{t-2}$
  cocircuits, which correspond to the 4-cycles in the dual graph.
 
  Quadrangulations are maximal planar bipartite graphs -- no edge can be
  added so that graph remains planar and bipartite. Using Euler formula for
  planar graphs \cite{nishizeki1988planar}, we obtain 
  $|E|=2|V|-4$. Theorem~\ref{thm:size}, furthermore, implies 
  $|E|=2t^2-2t$ and thus $|C|=|E|/t=2t-2$.
  \
\end{proof}

\begin{figure}
\begin{minipage}[b]{0.45\linewidth}
\centering
\includegraphics[width=\textwidth]{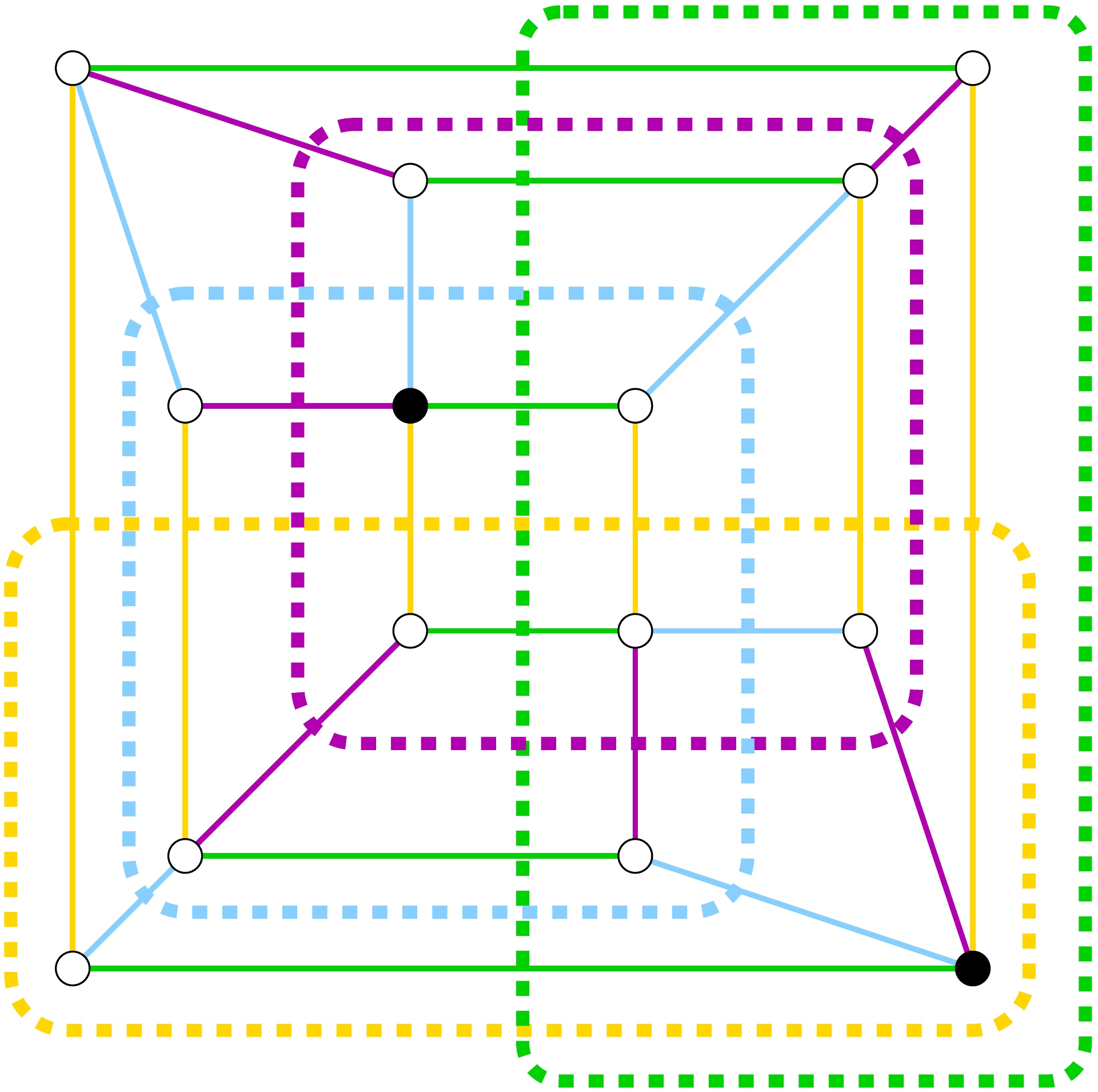}
\end{minipage}
\hspace{0.5cm}
\begin{minipage}[b]{0.45\linewidth}
\includegraphics[width=\textwidth]{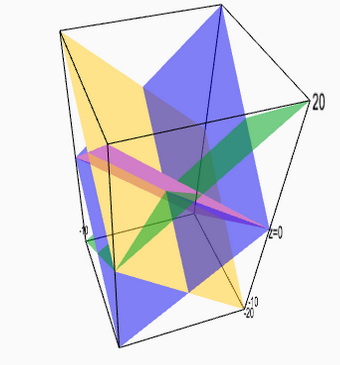}
\end{minipage}
\caption{Topological representation of rhombododecahedron (l.h.s.) in terms
  of its pseudocircle arrangement (doted curves) and the corresponding
  hyperplane arrangement (r.h.s.).}
\label{fig:toporep}
\end{figure}

As an example, Fig.~\ref{fig:toporep} shows the pseudocircle arrangement of
transit graph $R_2(\verb|----|,\verb|++++|)$ of
Fig.\ref{fig:rhombododecahedralgraph} and its equivalent hyperplane
arrangement.

In order to get a better intuition on the structure of the graphs induced
by 2-point crossover operators, we finally derive their degree sequence.
\begin{theorem}
  Let $d(a,b)=t > 3$. The degree sequence of a graph induced by $R_2(a,b)$
  equals $(t,t,4,\ldots,4,3,\ldots,3)$, where there are $t^2-3t$ vertices
  of degree $4$ and $2t$ vertices of degree $3$.
\end{theorem}
\begin{proof}
  W.l.o.g., let $a=0\ldots 0$ and $b=1 \ldots 1$. For any vertex $c=x
  \ldots xyx \ldots x$, $x,y \in \{0,1\}$ we have that $c \in R_2(a,b)$,
  hence $deg(a)=deg(b)=t$. Let $c \in R_2(a,b) \setminus \{a,b\}$. Then we
  have two cases:

  \textbf{Case 1.} $c= xx \ldots xxyy \ldots yy$ and $\{x,y\}=
  \{0,1\}$. Then $c$ has at most four neighbours in $R_2(a,b)$: $c_1=yx
  \ldots xxyy \ldots yy$, $c_2=xx \ldots xxyy \ldots yx$ , $c_3=xx \ldots
  xyyy \ldots yy$ and $c_4=xx \ldots xxxy \ldots yy$. Since $t > 3$ it
  follows $c$ also has at least three neighbours in $R_2(a,b)$.

  \textbf{Case 2.} $c= x \ldots xxyy \ldots yyxx \ldots x$ and $\{x,y\}=
  \{0,1\}$. Then $c$ has at most four neighbours in $R_2(a,b)$: $c_1=x
  \ldots xxxy \ldots yyxx \ldots x$, $c_2=x \ldots xyyy \ldots yyxx \ldots
  x$ , $c_3=x \ldots xxyy \ldots yxxx \ldots x$ and $c_4=x \ldots xxyy
  \ldots yyyx \ldots x$. Since $t > 3$ it follows $c$ also has at least
  three neighbours in $R_2(a,b)$.

  Let $x_3$ and $x_4$ denote the number of vertices of degree 3 and 4
  respectively. By the handshaking lemma $2|E| = \sum_{v \in V(G)}
  deg(v)$. Therefore, it follows from arguments above and
  Theorem~\ref{thm:R2} that
  \begin{align*}
    4t^2 -4t &= 2t + \sum_{v \in V(G) \setminus \{a,b\}} deg(v)\\
    4t^2 -6t &= 3x_3 + 4x_4
  \end{align*}
  Theorem \ref{thm:R2} also implies $t^2-t = x_3+x_4$. 
  Solving this system of linear equations yields 
  $x_3=2t$ and $x_4=t^2-3t$.
  \
\end{proof} 

\section{Concluding remarks}

Crossover operators are a key ingredient in the construction of algorithms
in Evolutionary and Genetic Programming. Their purpose is to construct
offsprings that are a ``mixture'' of the two parental genotypes, an idea
that is captured well by the concept of transit functions. In this
contribution we have investigated in detail the transit sets of homologous
crossover operators for strings of fixed length and their combinatorial,
graph theoretic, and topological properties.

As shown by \cite{Gitchoff:96} 1-point crossover operators correspond to
circles, that is, rather simple 2-dimensional objects. For $k>1$ we have
shown that $k$-point crossover operators are of more complex nature and
correspond to higher dimensional objects, which is appropriately measured
by the VC-dimension. For the case of binary alphabets, we explored the
close connection with oriented matroids and found that the elements of
transit sets form the topes. Furthermore, there is an equivalent
characterization in terms of pseudosphere arrangements. Since string
recombination operators effectively only distinguishes whether a sequences
position is equal or different between the two parents, the results on the
graph-theoretical structures of the transit rests directly carry over to
arbitrary Hamming graphs.

Linear oriented matroids are exactly those that can be represented by
sphere arrangements, i.e., every member of such arrangement is
$(d-1)$-dimensional sphere in $\mathbb{S}^d$, which is in turn equivalent
to the representation of oriented matroid by a central hyperplane
arrangement, e.g. their tope graphs are zonotopes. This suggest the
following open
\begin{problem}
  Are uniform oriented matroids corresponding to the $k$-point crossover
  operators realizable? In the case of positive answer, find equations
  describing the hyperplanes in the corresponding hyperplane arrangement.
\end{problem}

The results presented here also suggest to consider transit sets of
recombination operators for state spaces other than strings. Natural
candidates are many crossover operators for permutation problems. A subset
of these was compared e.g.\ by \cite{Puljic:13} and \cite{Bala:15} but
very little is known about the algebraic, combinatorial, and topological
properties. Interestingly, the $1$-point crossover operator $R_1$ satisfies
all axioms except (B2) of the axioms characterizing the interval function
of an arbitrary connected graph. Nevertheless, there are striking
differences even though both functions induce the same convexity as noted
in Lemma \ref{lem:hatR-I}.

Finally, recombination operators influence in a critical manner
they way how genetic information is passed down through the generalizations
in diploid populations. The corresponding nonassociative algebraic
structures so far have been studied mostly as generalizations of Mendel's
laws \cite{Bernstein:1923,Etherington:1939}, see also the books by
\cite{WorzBusekros:80} and \cite{Lyubich:92}. We suspect that a better
understanding of the structure of recombination operators will also be of
interest in this context.

\begin{acknowledgements}
  This work was supported in part by the Government of India, Department of
  Atomic Energy (grant no.\ 2/48(2)/2014/NBHM(R.P)-R\&D II/14364 to MC),
  and the Deutsche Forschungsgemeinschaft within the EUROCORES Programme
  EUROGIGA (project GReGAS) of the European Science Foundation (grant no.\
  STA 850/11-1 to PFS).
\end{acknowledgements}

\bibliographystyle{mde}      % Harvard style
\bibliography{recombin}

  \section*{Appendix A: Pseudosphere Arrangmements}
%\begin{appendix} 
  Consider the $d$-dimensional sphere $\mathbb{S}^d$ in $\mathbb{R}^{d+1}$
  and the corresponding $(d+1)$-dimensional ball
  $\mathbb{B}^{d+1}=\{(x_1,\ldots,x_{d+1}) \in \mathbb{R}^{d+1} \,|\, x^2_1
  + \ldots + x^2_{d+1} \leq 1\}$, whose boundary surface is $\mathbb{S}^d$.

  A pseudosphere $S \subset \mathbb{S}^d$ is a tame embedded
  $(d-1)$-dimensional sphere. Its complement in $B^d$ consist of exactly
  two regions, hence $S$ can be oriented, by labelling one region by
  $S^+_e$ and the other by $S^-_e$. A pseudosphere arrangement
  $\mathcal{S}=\{S_e \,|\, e\in E\}$ in the Euclidean space $\mathbb{R}^d$
  is a collection of $(d-1)$-dimensional pseudospheres on the
  $d$-dimensional unit sphere $\mathbb{S}^d$, where the intersection of any
  number of spheres is again a sphere and the intersection of an arbitrary
  collection of closed sides is either a sphere or a ball, i.e., for all $R
  \subset E$ holds
  \begin{enumerate}
  \item[(i)] $S_R=\mathbb{S}^d \cap_{i\in R} S_i$ is empty or homeomorphic
    to a sphere.
  \item[(ii)] If $e \in E$ and $S_R \not\subset S_e$ then $S_R \cap S_e$ is
    a pseudosphere in $S_R$, $S_R \cap S^+_e\ne\emptyset$ and $S_R \cap
    S^-_e\ne\emptyset$.
  \end{enumerate}
  For a pseudosphere arrangement $\mathcal{S}$, the position vector
  $\sigma(x)$ of a point $x \in \mathbb{S}^d$ is defined as
  \begin{equation*}
    \sigma(x)_e=
    \begin{cases}
      +, & \text{for } x \in S^+_e \\
      0, & \text{for } x \in S^+_e \\
      -, & \text{for } x \in S^-_e
    \end{cases}.
  \end{equation*}
  The set of all position vectors of $\mathcal{S}$ is denoted by
  $\sigma(\mathcal{S})$. 

  A famous theorem due to \cite{folkman1975oriented} establishes an
  correspondence between oriented matroids and pseudosphere arrangement. 
  
  \smallskip\par\noindent\textbf{Topological Representation Theorem}\quad
  \textit{Let $M=(E,\mathcal{F})$ be an oriented matroid of rank $d$. Then
    there exists a pseudosphere arrangement $\mathcal{S}$ in $\mathbb{S}^d$
    such that $\sigma(\mathcal{S})=\mathcal{F}$. Conversely, if
    $\mathcal{S}$ is a pseudosphere arrangement in $\mathbb{S}^d$, then
    $(E,\sigma(\mathcal{S}))$ is an oriented matroid of rank $d$.}  
  \medskip
  \par\noindent
  A simple alternative proof was given by \cite{bokowski2005topological}.

  A pseudosphere arrangement naturally induces a cell complex on
  $\mathbb{S}^d$, whose partial order of faces corresponds precisely to the
  partial order $\leq$ on covectors of the corresponding oriented matroid.
  This fact serves as motivation for concept of covectors in the theory of
  oriented matroids.

%\end{appendix}
\end{document}